\documentclass[11pt]{amsart}
\usepackage{amssymb}
\usepackage{amsmath}
\usepackage{amsthm}
\usepackage[textwidth=17cm, textheight=24cm, marginratio=1:1]{geometry}

\usepackage[english]{babel}
\usepackage[utf8]{inputenc}
\usepackage[T1]{fontenc}
\usepackage{todonotes}
\usepackage{xcolor}
\usepackage{tikz-cd}
\usetikzlibrary{calc}

\usepackage{hyperref}
\hypersetup{
    colorlinks,
    linkcolor={red!80!black},
    citecolor={blue!80!black},
    urlcolor={blue!80!black}
}
\usepackage{enumitem}
\setlist[enumerate,1]{label=(\arabic*), ref=(\arabic*), itemsep=0em}
\setlist[itemize,1]{topsep=0.5em}
\usepackage{setspace}
\numberwithin{equation}{section}
\newtheorem{theorem}{Theorem}[section]
\newtheorem{corollary}[theorem]{Corollary}
\newtheorem{lemma}[theorem]{Lemma}

\newtheorem{proposition}[theorem]{Proposition}
\theoremstyle{definition}
\newtheorem{definition}[theorem]{Definition}
\newtheorem{remark}[theorem]{Remark}
\newtheorem{example}[theorem]{Example}
\newtheorem{problem}[theorem]{Problem}

\DeclareMathOperator{\im}{im}
\DeclareMathOperator{\id}{id}
\DeclareMathOperator{\Spec}{Spec}
\DeclareMathOperator{\Proj}{Proj}

\DeclareMathOperator{\Hom}{Hom}
\DeclareMathOperator{\Ext}{Ext}
\newcommand{\tensor}{\otimes}
\newcommand{\onto}{\twoheadrightarrow}
\newcommand{\into}{\hookrightarrow}
\newcommand{\hook}{\,\lrcorner\,}
\newcommand{\kk}{\Bbbk}%
\newcommand{\kadp}{\kappa_{\mathrm{dp}}}%
\newcommand{\kappadp}{\kappa_{\mathrm{dp}}}%
\newcommand{\spann}[1]{\left\langle #1 \right\rangle}
\newcommand{\sat}{\operatorname{sat}}%
\newcommand{\redd}{\operatorname{red}}%
\DeclareMathOperator{\Hilb}{Hilb}%
\DeclareMathOperator{\Sat}{Sat}%
\DeclareMathOperator{\Slip}{Slip}
\DeclareMathOperator{\pr}{pr}%
\DeclareMathOperator{\Gr}{Gr}%
\newcommand{\OO}{\mathcal{O}}%
\newcommand{\sm}{\operatorname{sm}}%
\DeclareMathOperator{\brr}{br}
\DeclareMathOperator{\srr}{sr}
\DeclareMathOperator{\rr}{r}
\DeclareMathOperator{\crr}{cr}
\newcommand{\genHF}[2]{H_{#1,\mathbb{P}^{#2}}}
\newcommand{\genMHS}[2]{\Hilb_{#1, \mathbb{P}^{#2}}}
\newcommand{\genSat}[2]{\Satbar_{#1,\mathbb{P}^{#2}}}

\newsavebox{\pullback}
\sbox\pullback{%
\begin{tikzpicture}%
\draw (0,0) -- (1ex,0ex);%
\draw (1ex,0ex) -- (1ex,1ex);%
\end{tikzpicture}}
\DeclareMathOperator{\TgFlag}{Tg_{\,\mathrm{flag}}}
\DeclareMathOperator{\ObFlag}{Ob_{\,\mathrm{flag}}}
\DeclareMathOperator{\ObFib}{Ob_{\,\mathrm{fiber}}}
\DeclareMathOperator{\ObSymbol}{Ob}
\DeclareMathOperator{\GL}{GL}
\newcommand{\mm}{\mathfrak{m}}%
\newcommand{\pp}{\mathfrak{p}}%
\renewcommand{\aa}{\mathbf{a}}%
\newcommand{\bb}{\mathbf{b}}%
\newcommand{\dd}{\mathbf{d}}%
\newcommand{\ee}{\mathbf{e}}
\newcommand{\vv}{\mathbf{v}}
\newcommand{\midx}{\mathbf{m}}%
\newcommand{\BBname}{Bia{\l{}}ynicki-Birula}
\newcommand{\zerodeg}{\mathbf{0}}

\newcommand{\Irr}{\operatorname{Irr}}%
\renewcommand{\AA}{\mathcal{A}}%

\DeclareMathOperator{\Pic}{Pic}
\DeclareMathOperator{\Nef}{Nef}
\DeclareMathOperator{\Eff}{Eff}
\DeclareMathOperator{\Ann}{Ann}

\begin{document}

\title{Limits of saturated ideals}
\author{Joachim Jelisiejew, Tomasz Ma{\'n}dziuk}
\thanks{JJ is supported by National Science Centre grant 2020/39/D/ST1/00132.
TM was supported by National Science Centre grant 2019/33/N/ST1/00858.
This work is partially supported by the
Thematic Research Programme \emph{Tensors: geometry, complexity and quantum
entanglement}, University of Warsaw, Excellence Initiative – Research
University and the Simons Foundation Award No.~663281 granted to the Institute
of Mathematics of the Polish Academy of Sciences for the years 2021-2023.}
\begin{abstract}
    We investigate the question of when a given homogeneous ideal is a limit
    of saturated ones. We provide cohomological necessary criteria for this to
    hold and apply them to a range of examples. In small cases we characterise the
    limits. We also supply a number of auxiliary results on the classical and
    multigraded Hilbert schemes, for example we prove a very general result on
    openness of the saturated locus and give scheme structure to ``constant
    Hilbert function'' loci. Our results have applications to theory of
    tensors.
\end{abstract}
\date{\today}
\maketitle

\section{Introduction}

    It is famously untrue that a deformation of a zero-dimensional projective
    scheme $\Gamma \subseteq \mathbb{P}^{n}$ induces a deformation of its
    homogeneous coordinate ring $S/I(\Gamma)$, where
    $I(\Gamma)\subseteq S =\kk[\alpha_0, \ldots,
    \alpha_n]$ is the saturated ideal. The standard example is given by the
    varying three points
    \[
        \Gamma_t = \{[1:0:0], [1:t:1], [0:0:1]\}\subseteq \mathbb{P}^2,
    \]
    which form a flat family over the affine line with parameter $t$. The ideal
    $I(\Gamma_t|_{t=\lambda})$
    contains a linear form exactly when $\lambda = 0$, so
    if the family $S[t]/I(\Gamma_t)$ existed, it had to be \emph{lower} semi-continuous
    in degree one; an absurd. What happens instead is that as $\Gamma_t$
    converges towards $\Gamma_0$, the family of
    ideals $I(\Gamma_t|_{t=\lambda})$ converges degreewise to a \emph{nonsaturated} ideal
    \[
        I(\Gamma_0)' = \left(\alpha_0\alpha_2(\alpha_0 -
        \alpha_2)\right) + \alpha_1\cdot\left(\alpha_0, \alpha_1, \alpha_2
        \right).
    \]
    For a homogeneous ideal $I$ we say that it is \emph{a limit of saturated
    ideals} or \emph{saturable} if there is a family of saturated ideals
    that converges to $I$. Hence $I(\Gamma_0)'$ above is saturable, yet not
    saturated. In this article we tackle the following.
    \begin{problem}\label{ref:mainProblem}
        Given a homogeneous ideal, decide whether it is saturable.
    \end{problem}
    There are numerical obstructions to being saturable: an ideal has to have
    an admissible Hilbert function.
    In the case of a standard graded polynomial ring such functions are
    classified by Macaulay~\cite[\S4.2]{BrunsHerzog}.
    However, once the Hilbert
    function is admissible, there are no general
    criteria for verifying whether an ideal is saturable. We give such
    criteria below, using deformation theory.

    \subsection{Results}
    Our results are of three kinds. First, we prove
    several results for classical and multigraded Hilbert schemes: see
    Theorem~\ref{ref:intro:saturated},
    Theorem~\ref{ref:opennessOfSaturation:thm} and
    Theorem~\ref{ref:obstructionAtFlag:thm}. Second, we employ them and our
    idea of ``stickyness'' to give
    explicit necessary criteria for being saturable. These
    criteria are checkable by computer on specific examples, so
    they are useful also for researchers outside algebraic geometry.
    Third, as an illustration, we give examples and applications to specific
    low degrees and to wild polynomials (see~\S\ref{sec:wild}).

    Below, we consider (multi)homogeneous ideals but we do not impose any further
    restrictions, for example, we do \emph{not} restrict to the zero-dimensional
    case. Actually, we take care to prove results under minimal assumptions:
    we allow any multigrading, any dimension,
    any base ring, see~\S\ref{ssec:setup} for details.


        Before we present our results, we informally present the observation which
        underlies our criteria. Let $S$ be a Cox ring of a smooth projective toric variety. In
        algebraic terms, this implies that $S$ is a polynomial ring graded by a
        finitely generated abelian
        group $\AA$ and admitting an irrelevant ideal $\Irr$ with respect to
        which we consider the saturation (see
        Example~\ref{ex:productOfProjectives} for one concrete instance of
        this setup). We say that
        an ideal $J$ with $I \subsetneq J\subseteq I^{\sat}$ \emph{sticks with}
        $I$ (or is a \emph{sticky ideal} for $I$) if every deformation of $I$ induces a deformation of
        $J\supseteq I$.  If $I$ admits a sticky ideal $J$, then it is not a limit of saturable
        ideals. Indeed, suppose $I_t\to I$ with $I_t|_{t=\lambda}$ saturated
        for $\lambda\neq 0$.
        Then we have a corresponding family $J_t\supsetneq I_t$. Due to
        equal Hilbert polynomials, the ideals $J_t|_{t=\lambda}$ and
        $I_t|_{t=\lambda}$ agree in ``positive enough''
        degrees, so $(J_t|_{t=\lambda})^{\sat} = (I_{t}|_{t=\lambda})^{\sat} = I_t|_{t=\lambda}$, in particular
        $J_t|_{t=\lambda}\subseteq I_t|_{t=\lambda}$, which is a contradiction.

        In this paper, we obtain a toolbox of cohomological criteria for proving that a given ideal
        $J$ sticks with $I$. The \emph{fiber obstruction group} for the pair
        $(I, J)$ is
        \begin{equation}\label{eq:introOb}
            \ObFib(I,J) := \Ext^1_S\left( \frac{J}{I},\ \frac{S}{J}
            \right)_{\zerodeg},
        \end{equation}
        where $\zerodeg\in \AA$ is the zero element.
        The most succinct version of our criterion is as follows.
        \begin{theorem}[Theorem~\ref{ref:deformsWhenObFibVanishes:thm}]\label{ref:introObVanishes:thm}
            If $\ObFib(I, J) = 0$, then $J$ sticks with $I$, hence $I$ is not
            saturable.
        \end{theorem}
        This is effective already for the ideal $J :=
        I^{\sat}$, see~\S\ref{ssec:four}.
        A slightly changed version is the following, which employs the
        multigraded Hilbert scheme, see~\S\ref{ssec:setup}.
        \begin{theorem}[Theorem~\ref{ref:deformWhenTangentSurjective:thm}]\label{ref:introSmoothAndTangentSpace:thm}
            If the natural map $\Hom_S(J, S/J)_{\zerodeg}\to\Hom_S(I,
            S/J)_{\zerodeg}$ is onto
            and $[J]\in \Hilb^{H_{S/J}}$ is smooth, then $J$ sticks with $I$, hence $I$ is not
            saturable.
        \end{theorem}

        We say that $I$ is \emph{entirely nonsaturable} if it is nonsaturated and
        there is an open neighbourhood of $[I]\in \Hilb^{H_{S/I}}$ such that any ideal
        $I'$ from this neighbourhood has $S/(I')^{\sat}$ with the same Hilbert
        function as $S/I^{\sat}$, so ``every near enough ideal is as nonsaturated
        as $I$''. If the conditions of
        Theorem~\ref{ref:introObVanishes:thm} or
        Theorem~\ref{ref:introSmoothAndTangentSpace:thm} are satisfied for $J :=
        I^{\sat}$, then $I$ is entirely nonsaturable (see
        Theorems~\ref{ref:saturationDeformsWhenObFibVanishes:thm}-\ref{ref:saturationDeformWhenTangentSurjective:thm}).

        The necessary conditions for both theorems above are easily computable
        with computer systems such as \emph{Macaulay2}~\cite{M2}. However,
        the presentation of fiber obstruction~\eqref{eq:introOb} is opaque.
        There is a more transparent version when $\AA = \mathbb{N}$ and
        $J = I^{\sat}$ defines a one-dimensional subscheme.
        Namely, the vector space $\ObFib$ is dual to $(I^{\sat}/I\tensor_{S}
        \omega_{S/I^{\sat}})_0$, where $\omega_{S/I^{\sat}}$ is the canonical
        module, see Proposition~\ref{ref:CIAndcanonicalModule:prop}.
        This reinterpretation is especially useful when we assume
        additionally that $S/I^{\sat}$ is Gorenstein (which holds for example if
        $I^{\sat}$ is a complete intersection). In this case, we have
        \begin{equation}\label{eq:introObEasy}
            \ObFib(I, I^{\sat})  \simeq
            \left(\frac{I^{\sat}}{I+(I^{\sat})^2}\right)_{a}^{\vee}
        \end{equation}
        where $a$ is the largest degree for which
        $H_{S/I^{\sat}}(a)
        <H_{S/I^{\sat}}(a+1)$ see~Corollary~\ref{ref:fiberObstructionForGorenstein:cor}. The obstruction group can be more
        generally computed for non-Gorenstein $J=I^{\sat}$ if additional
        conditions on the Hilbert functions are satisfied, see
        Example~\ref{ex:lineAndPoints}.

        Above we discussed \emph{necessary} conditions for being saturable. It
        is natural to ask what happens if those are not satisfied.
        In this case, interestingly, we do get information about possible
        families showing saturability of $I$. For example, in the context
        of Theorem~\ref{ref:introSmoothAndTangentSpace:thm} for
            $J=I^{\sat}$, if $\Hom_S(J,
            S/J)_{\zerodeg}\to \Hom_S(I, S/J)_{\zerodeg}$ is not surjective and yet $\Hom_S(I,
            S/I)_{\zerodeg}\to \Hom_S(I, S/J)_{\zerodeg}$ is surjective, then we obtain a tangent
        vector (see~\S\ref{ssec:maindiagram}) at $[I]$, which
        does not come from a tangent vector at $[I\subseteq I^{\sat}]$. If we are able to
        integrate along it (for example, if $[I]$ is smooth), then we obtain a
        one-dimensional family whose
        general member $I'$ has $H_{S/(I')^{\sat}} > H_{S/I^{\sat}}$, where
        the inequality means argument-wise inequality with at least one of
        them strict. Repeating this process we arrive at the case where
        $H_{S/(I')^{\sat}} = H_{S/I} = H_{S/I'}$, hence $I'$ is saturated.
        This procedure seems effective on examples even when $[I]$ is not
        smooth.

        \subsection{Examples} To illustrate the effectiveness of our ideas, we
        decide the saturability for all ideals with Hilbert function $(1,d,d,d,
        \ldots )$ for $d\leq 5$; in this case $\AA = \mathbb{N}$, see~\S\ref{sec:examples}. For $d = 4$ and
        all but one cases for $d = 5$ a nonsaturated ideal $I$ is saturable  if and only
        if the obstruction space~\eqref{eq:introObEasy} is nonzero.
        The unique case for $d = 5$ is when $H_{S/I^{\sat}} = (1,2,3,4,5,5,
        \ldots )$. In this case the nonvanishing of the obstruction boils down
        to $(I^{\sat})^2_3\subseteq I$. It is still
        necessary, yet not sufficient: there is a second necessary condition
        of the form $(I^{\sat}\cdot J)_2\subseteq I$ for a certain
        ideal $I\subseteq J \subseteq I^{\sat}$ determined by $I$.
        The same argument would apply to other Hilbert functions
        eventually equal to $d\leq
        5$; we have restricted ourselves to $(1,d,d, \ldots )$ in order to
        better exhibit the main ideas.

    \subsection{Application to the classical Hilbert scheme} While our focus is on nonsaturated ideals,
    we obtain two results about saturated ideals. They are of
    independent interest.

    The first results concerns the openness of saturation.
    \begin{theorem}[openness of saturation,
        Theorem~\ref{ref:opennessOfSaturation:thm}]\label{ref:openness:introthm}
        For any $\AA$-grading on $S$, any Hilbert function
        $H\colon \AA\to \mathbb{N}$, and any fixed homogeneous ideal $\Irr$, the subset of $[I]\in \Hilb^H$
        which are saturated (with respect to $\Irr$) is open.
    \end{theorem}
    The above seems not present in the literature even in the case of standard
    $\mathbb{N}$-grading on the polynomial ring $S$ and $\Irr = S_{\geq0}$. It is important in itself, but also for applications to tensors, where it
    allows to define the Slip component(s), see
    Definition~\ref{ref:Slip:def}.

    Our second result gives a functorial description of the set of all
    zero-dimensional ideals with given Hilbert function.
    \begin{theorem}[Theorem~\ref{ref:partialToGrothendieck:prop}]\label{ref:intro:saturated}
        Let $S$ be the Cox ring of a smooth projective toric variety $X$ and $\AA \simeq
        \Pic(X)$ be the
        induced grading group.
        For any function $H\colon \AA\to \mathbb{N}$ with multigraded
        Hilbert polynomial equal to $d\in \mathbb{N}$,
        the locus of $[\Gamma]\in \Hilb^d (X)$ with
        $H_{S/I(\Gamma)} = H$ is locally closed in $\Hilb^d
        (X)$ and admits
        a scheme structure with which it is isomorphic to $\Hilb^{H,\sat}$, the saturated locus
        of $\Hilb^{H}$.
    \end{theorem}
    Theorem~\ref{ref:intro:saturated} allows in particular for a effective
    calculation of the tangent space to the locus.

    In the case of projective space, we can say more about the smoothness.
    \begin{theorem}[Theorems~\ref{ref:smoothPoints:thm}-\ref{ref:usualClassesSmooth:cor}]\label{ref:intro:saturatedSmooth}
        In the setting of Theorem~\ref{ref:intro:saturated}, suppose
        additionally that $X = \mathbb{P}^{n-1}$. Then, a point $[\Gamma]$ in
        the saturated locus is smooth whenever the Artin reduction of $S/I(\Gamma)$ is
        unobstructed. In particular, $[\Gamma]$ is smooth when $n \leq 3$ or
        when $n \leq 4$ and $S/I(\Gamma)$ Gorenstein.
    \end{theorem}
    In the case $n = 3$ and for a field of characteristic zero Gotzmann proved
    that the above locus is smooth~\cite{gotzmann_fixed_Hilbert_function}.
    Kleppe proved smoothness in the Gorenstein $n=4$
    case~\cite{kleppe__smoothness}. The
    interpretation as a functor seems new even in these special cases.

    \subsection{Application to nonexistence of wild polynomials}

        We apply the above results also to the theory of border rank in the
        case of $\mathbb{P}^2$.
        It is an open question whether wild
        forms (see~\S\ref{sec:wild}) in three variables exist. By~\cite[Proposition~3.4]{Mandziuk} degree $d$
        forms of border rank at most $d+2$ are not wild. Here we improve the
        bound by one.
        \begin{proposition}[Proposition~\ref{prop:no_wild_polynomials}]
            Let $F$ be a degree $d$ form in three variables. If the border
            rank of $F$ is at most $d+3$, then the smoothable rank of $F$ is
            equal to its border rank, so $F$ is not wild.
        \end{proposition}
        There are lower bounds for the smoothable rank such
        as~\cite{ranestad_schreyer_on_the_rank_of_a_symmetric_form}.
        Coupled with the above they yield constraints on the possible forms of
        low border rank. To give a concrete example, a septic with Hilbert
        function $(1,3,6,9,9,6,3,1)$ and border rank $9$ has, by the above,
        smoothable rank $9$ and so,
        by~\cite[Proposition~1]{ranestad_schreyer_on_the_rank_of_a_symmetric_form},
        its apolar ideal cannot be generated in degrees at
        most $4$. This also follows from border apolarity \cite[Theorem~3.15]{Buczyska_Buczynski__border}.

    \subsection{Motivation from theory of VSP} Investigating Problem~\ref{ref:mainProblem} is
    very natural from the commutative algebra perspective.
    But, for us, the main motivation comes from the theory of tensors.
    A central problem there is to determine the border rank of a tensor (see
    Section~\ref{sec:wild} for a brief introduction). This problem is solved
    in very very few cases and is open even for tensors of small
    size~\cite[\S16]{burgisser_clausen_shokrollahi},
    \cite{Conner_Huang_Landsberg}.

    A recent successful technique is \emph{border
    apolarity}~\cite{Buczyska_Buczynski__border} (see also
    \cite{ranestad_schreyer_VSP, Ranestad_Iliev__VSPcubic,
        Ranestad_Schreyer__polar_simplices,
    Ranestad_Voisin__VSP, Jelisiejew_Ranestad_Schreyer}). To a tensor $T$ and
    integer $r$ it assigns a projective scheme, the \emph{(border) variety of sums of
    powers} whose points are certain multigraded
    ideals $I$ contained in $\Ann(T)$. The tensor $T$ has border rank at most
    $r$ if and only if the scheme contains a saturable ideal with multigraded Hilbert polynomial
    $r$ which is a limit of reduced ones. Frequently, even
    if such an ideal exists, the scheme contains no saturated ideals at all.
    Producing ideals in the scheme is in practice easy, but deciding whether
    they are saturable is hard.

    \subsection{Previous work}

        To our best knowledge, there are almost no papers concerned with
        Problem~\ref{ref:mainProblem}. This problem is mentioned
        in~\cite{Buczyska_Buczynski__border, Conner_Huang_Landsberg}. One may also ask 
        which ideals are limits of saturated and radical ideals. 
        The latter problem was investigated in~\cite{Mandziuk} in the $\mathbb{N}$-graded case.
        It gives two necessary conditions: a semicontinuity of tangent spaces type of statement
        \cite[Theorem~1.1]{Mandziuk} and a necessary condition in the case that $V(I)$ is contained in the 
        line~\cite[Theorem~2.7]{Mandziuk}. The latter statement is an immediate corollary of
        Theorem~\ref{ref:introObVanishes:thm} above.

        Theorem~\ref{ref:saturationDeformsWhenObFibVanishes:thm} is
        inspired by the works of Cid-Ruiz and
        Ramkumar~\cite{CidRuiz_Ramkumar__Fiber_full,
        CidRuiz_Ramkumar__Fiber_full_local} which in turn were inspired
        by~\cite{Conca_Varbaro__squareFree}. The main difference is that
        Cid-Ruiz and Ramkumar employ local cohomology and
        impose stronger conditions that guarantee that also the higher local
        cohomology groups $H^i$ deform with the ideal, while we only investigate
        whether $H^0_{\Irr}(S/I) = I^{\sat}/I$ deforms with $I$.
        This is especially important in the multigraded case, where
        $S/I^{\sat}$ has high dimension and need not even be Cohen-Macaulay,
        so intermediate cohomology may well be nonzero and not deform.
        Finally, our setting does not require standard $\mathbb{Z}$-grading and
            furthermore, by allowing intermediate ideals $J$ lying beetween
            $K$ and $K^{\sat}$ we obtain more flexibility. This is
            illustrated by Examples~\ref{exa:Z2_grading}-\ref{exa:intermediate_ideal}.

        Buczy{\'n}ska-Buczy{\'n}ski in an unpublished
        note~\cite{jabuSaturationOpen} gave a proof of a slightly weaker version
        of Theorem~\ref{ref:openness:introthm} based on the ideas of the
        first-named author.

\subsection*{Acknowledgements}

    We thank Alessandra Bernardi and Joseph Landsberg
    for helpful comments and conversations. We thank an anonymous referee for
    constructive suggestions regarding the presentation. We especially thank
    Jaros{\l}aw Buczy{\'n}ski for his very helpful comments.

\section{General results on multigraded ideals and Hilbert schemes}

Throughout the article, $\kk$ is a Noetherian ring of finite dimension. (On the first reading, any willing reader can assume
that $\kk := \kappa$ is their favourite field.)
The article makes frequent use of deformation theory. One source well aligned
with our results is~\cite[Chapter~5]{fantechi_et_al_fundamental_ag}. This
reference assumes that everything takes place over a field, but the proofs do
generalize in a straightforward way. For convenience of reader, we phrase
the main theorem of obstruction calculus explicitly as
Lemma~\ref{ref:obstructionCalculus:lem}.

\subsection{Conventions and notation}\label{ssec:setup}

    We fix a polynomial ring $S$ over $\kk$. For a $\kk$-algebra $A$ or a
    $\kk$-scheme $U$ we write $S_A$, $S_U$ instead of $S\otimes_{\kk} A$,
    $S\otimes_{\kk} \OO_U$, respectively.

        We assume that $\AA$ is a finitely generated abelian
        group and that $S$ is $\AA$-graded via a homomorphism
        $\mathbb{N}^{\#\mathrm{vars}(S)} \to \AA$. A \emph{homogeneous ideal} is an ideal
        homogeneous with respect to the $\AA$-grading.
        We remark that even when $\kk = \kappa$ is a field it may well happen that for some $\aa\in \AA$ the dimension
        of the $\kappa$-vector space $S_{\aa}$ is infinite (for example, take $S
        = \kappa[x, y]$, $\AA = \mathbb{Z}$, $\deg(x) = 1$, $\deg(y) = -1$, $\aa =
        \zerodeg$). This forces us to take care when defining familiar
        notions such as Hilbert functions.

        For any $\AA$-graded $S$-algebra $S'$ and any point $\Spec(L)\to
        \Spec(\kk)$, where $L$ is a field, the \emph{fiber} of $S'$ over $L$
        is $S'\otimes_{\kk} L$.
        If for every $\aa\in \AA$ the vector space $\left(S'\otimes_{\kk} L\right)_{\aa}$ has
        finite dimension, then the \emph{Hilbert function} of the
        fibre $H_{S'\otimes_{\kk} L}\colon \AA\to \mathbb{N}$ is given by
        \[
            H_{S'\otimes_{\kk} L}(\aa) = \dim_{L} \left( S'\otimes_{\kk} L \right)_{\aa}.
        \]
            If $\kk$ is a field, this
            definition agrees with the usual definition.

    We recall the terminology from~\cite{Haiman_Sturmfels__multigraded}.
    Fix a function $H\colon \AA \to \mathbb{N}$.
    For a $\kk$-scheme $U$ and a $\AA$-graded sheaf of ideals $I\subseteq
    S_U$, we
    say that $I$ is \emph{admissible with (finite) Hilbert function} $H$
    if for every $\ee\in \AA$ the $\OO_U$-module $\left((S_U)/I\right)_{\ee}$ is locally free of rank $H(\ee)$. We call a
    closed subscheme $\mathcal{Z} \subseteq \Spec(S)\times U$
    \emph{admissible with (finite) Hilbert function $H$} if its ideal
    sheaf is admissible with Hilbert function $H$.
    The multigraded Hilbert scheme $\Hilb^H$ is the scheme representing
    admissible ideals $I$.

    \newcommand{\Sk}{S_{\kappa}}
    Throughout the article, the symbol $\kappa$ will denote any field
    with a homomorphism $\kk\to\kappa$. The ideals of interest will lie in
    some $\Sk = S\otimes_{\kk} \kappa$.  Below we will use $\Sk$ without explicitly
    repeating that $\kappa$ is a field as above.
    For an admissible ideal
    $J\subseteq \Sk$ we define the shorthand notation $\Hilb_{J}$ for the $\Bbbk$-scheme
    $\Hilb^{H_{\Sk/J}}$.

    Likely, many readers are
    interested in the case where $\kk$ itself is a field. They may safely take
    $\kappa = \kk$, so that $\Sk = S$. The more general case is useful for
    example when considering deformations (with $\kk$ Artinian) and arithmetic
    questions (with $\kk = \mathbb{Z}$).

    \subsection{Saturations}\label{ssec:saturations}

        We fix a homogeneous ideal $\Irr\subseteq S$ which we will call the
        \emph{irrelevant ideal}.
        For a homogeneous ideal $I$ in an $\AA$-graded Noetherian $S$-algebra $S'$, its \emph{saturation} is the ideal
        \[
            I^{\sat} := \bigcup_{i\geq 0} (I : \Irr^i).
        \]
        Since $S'$ is Noetherian, we have $I^{\sat} = (I : \Irr^e)$ for all
        $e\gg 0$.
        The ideal $I$ is \emph{saturated} if $I = I^{\sat}$.
        The saturation of any
        ideal is saturated. Let $\Irr' := \Irr\cdot S'$. By prime
        avoidance~\cite[\href{https://stacks.math.columbia.edu/tag/00LD}{00LD},
        \href{https://stacks.math.columbia.edu/tag/00DS}{00DS}]{stacks-project},
        a homogeneous ideal $I\subseteq S'$ is saturated if
        an only if there exists a homogeneous $\ell\in \Irr'$ such that the multiplication
        by $\ell$ is injective on $S'/I$. For a homogeneous ideal $I\subseteq
        S'$ a
        \emph{transverse element} is any homogeneous
        $\ell\in \Irr'$ which is a nonzerodivisor on $S'/I^{\sat}$. For every
        such $\ell$, the ideal
        $I^{\sat}/I$ is the kernel of the localisation map $S'/I\to (S'/I)_{\ell}$.
        If $e$ is any integer such that $I^{\sat} = (I : \Irr^e)$, then
        $\ell^e$ is another transverse element and this element annihilates $I^{\sat}/I$.
        (In this article, the only $S'$ we consider are of the form
        $S_A$ for $A$ a $\kk$-algebra, trivially $\AA$-graded.)

        Proposition~\ref{ref:transverseElementDeforms:prop} is the technical part of the proof that saturation is an open
        property, see Theorem~\ref{ref:opennessOfSaturation:thm} below for
        motivation.

        \begin{lemma}\label{ref:Gotzmann:lem}
            Let $U$ be a Noetherian $\kk$-scheme of finite dimension and $Z \subset
            \Spec(S) \times U$ be a closed subscheme given by a
            homogeneous ideal sheaf. Denote by $p\colon Z\to U$ the projection.
            For every $u\in U$ consider the ``Hilbert function'' $H_u\colon
            \AA\to \mathbb{N} \cup \left\{ \infty \right\}$ defined
            by
            \[
                H_{u}(\aa) := \dim_{\kappa(u)} H^0(\OO_{p^{-1}(u)})_{\aa}.
            \]
            Then the set $\{H_u\ |\ u\in U\}$ is finite.
        \end{lemma}
        \begin{proof}
            We do induction on the dimension of $U$. For a given $U$, we
            replace it with $U_{\redd}$ and then by the disjoint union of the
            irreducible components of its
            affine open subschemes. For every such component $\Spec(A)$, let
            $\Spec(B) = p^{-1}(\Spec(A)) \subseteq Z$.

            By Generic Flatness
            (\cite[\href{https://stacks.math.columbia.edu/tag/051R}{Tag
            051R}]{stacks-project} or \cite[Theorem~14.4]{Eisenbud})
            there is
            a nonzero element $a\in A$ such that $B_{a}$ is a free
            $A_{a}$-module.

            Choose a degree $\ee\in \AA$. The $A_a$-module $(B_a)_{\ee}$ is a
            direct summand of $B_a$, hence is a projective $A_a$-module.
            Choose a point $\pp\in \Spec(A_a)$. The module
            $((B_a)_{\pp})_{\ee}$ is a projective $(A_a)_{\pp}$-module, hence
            is free
            by~\cite[\href{https://stacks.math.columbia.edu/tag/0593}{Tag
            0593}]{stacks-project}. Recall that $\Spec(A)$ corresponds to an
            irreducible scheme, so it has a generic point $\eta$. The points
            $\pp$, $\eta$ lie in $\Spec((A_{a})_{\pp})\subseteq \Spec(A)$. The
            freeness of $((B_a)_{\pp})_{\ee}$ implies that the Hilbert
            functions of $p^{-1}(\pp)$ and $p^{-1}(\eta)$ agree at position
            $\ee$. The point $\pp$ and the position $\ee\in\AA$ are arbitrary,
            so this means that the Hilbert function of $p^{-1}(u)$ is
            independent of the choice of $u\in (a\neq0)\subseteq
            \Spec(A)$.

            Repeating this argument for every component of $U$, we
            obtain an open dense subset $U'\subseteq U$ with only finitely many possible
            Hilbert functions, at most as many as the number of
            components.
            By induction, the set of Hilbert functions for the fibers over
            $U\setminus U'$ is finite.  This concludes the proof.
        \end{proof}
        \begin{proposition}\label{ref:transverseElementDeforms:prop}
            Let $f\in S$ be a homogeneous element. Let $U$ be a
            locally Noetherian $\kk$-scheme. Assume that $U$ has an open cover
            by schemes of finite dimension.
            Let $Z \subset \Spec (S) \times U$ be an admissible closed
            subscheme with Hilbert
            function $H$. Let $p\colon Z\to U$ be the projection.
            Then the subset
            \begin{equation}\label{eq:goodfibers}
                U' := \left\{ u\in U\ |\ f \mbox{ is a nonzerodivisor on
                } H^0(\OO_{p^{-1}(u)}) \right\} \subseteq U
            \end{equation}
            is open. The multiplication by $f$ on
            $\OO_{p^{-1}(U')}$ is injective and $V(f)\cap p^{-1}(U')$ is flat
            over $U'$.
        \end{proposition}
        \begin{proof}
            \def\HilbFunc{\mathcal{H}}%
            \def\II{\mathcal{I}}%
            Everything is local on $U$, so we assume $U =
            \Spec(A)$ is Noetherian of finite dimension.

            We have $Z =
            \Spec(B)$ for some $A$-algebra $B$.
            Pick a point $u_0\in U$, let $\pp \subset A$ be the prime ideal
            corresponding to $u_0$ and $\kappa_0 = A_{\pp}/\pp A_{\pp}$. Suppose that $f$ is a
            nonzerodivisor on $B\otimes_A \kappa_0 = H^0(\OO_{p^{-1}(u_0)})$.
            Let $T = A\setminus
            \pp$ be the multiplicatively closed subset, so that $T^{-1}A =
            A_{\pp}$.
            Let $\dd = \deg(f)$ and let $\mu_f\colon B\to B$ be the multiplication
            by $f$. Fix any element $\ee\in\AA$.
            Consider the map $(\mu_f)_{\ee}\colon B_{\ee}\to B_{\dd+\ee}$. By assumption, the
            map
            \[
                (\mu_f)_{\ee}\otimes_{A} \kappa_0\colon B_{\ee} \tensor_A \kappa_0 \to B_{\dd+\ee} \tensor_A
                \kappa_0
            \]
            is injective. By~\cite[Lemma~A.7]{Sernesi__Deformations} the map
            $T^{-1}(\mu_f)_{\ee}\colon T^{-1}B_{\ee}\to T^{-1}B_{\dd+\ee}$
            of $A_{\pp}$-modules is split injective; in particular its
            cokernel $T^{-1}(B_{\dd+\ee})/fT^{-1}B_{\ee}$ is a free
            $A_{\pp}$-module
            because it is a direct summand of the free $A_{\pp}$-module
            $T^{-1}B_{\dd+\ee}$.
            Since localization is exact, the
            module $T^{-1}(B_{\dd+\ee})/fT^{-1}B_{\ee}$ is
            isomorphic to $(T^{-1}(B/fB))_{\dd+\ee}$, so the latter is a free
            $A_{\pp}$-module as well.
            This holds for any $\ee\in\AA$, so, by
            reindexing, we conclude that $(T^{-1}(B/fB))_{\ee}$ is a free
            $A_{\pp}$-module for every $\ee\in \AA$.

            Consider the degree decomposition
            \[
                T^{-1}(B/fB) = \bigoplus_{\ee\in \AA} \left( T^{-1}(B/fB)
                \right)_{\ee}.
            \]
            The summands are finitely generated free $A_{\pp}$-modules, so
            $T^{-1}(B/fB)$ is a free $A_{\pp}$-module (not necessarily
            finitely generated!).
            We now prepare to ``smear out'' the freeness of $T^{-1}(B/fB)$ to
            the freeness of some $(B/fB)_s$, $s\in T$.
            Each individual
            factor can be ``smeared out'' to a free $A_s$-module for some
            $s\in T$, see~\cite[\href{https://stacks.math.columbia.edu/tag/00NX}{Tag 00NX}]{stacks-project}.
            The issue is that the direct sum has possibly an infinite number
            of nonzero summands, so we need to take care to find a common $s$.

            Let $H'$ be the Hilbert function of $(B/fB)\tensor_A \kappa_0$,
            so that for every $\ee\in \AA$ we have
            \begin{equation}\label{eq:Hprim}
                H'(\ee) = H(\ee) - H(\ee-\dd).
            \end{equation}
            By Lemma~\ref{ref:Gotzmann:lem}, the set $\HilbFunc$ of possible Hilbert
            functions for the fibers of $B/fB$ is finite. Fix a finite subset
            $\II\subseteq \AA$ such that
            every Hilbert function $H''\in \HilbFunc\setminus \{H'\}$ differs from
            $H'$ on some position $\ee\in \II$.
            Fix $s\in T$ such that $(B_{\ee}/fB_{\ee-\dd})_s$ is a free $A_s$-module for every
            $\ee \in \II$.

            For every $u\in \Spec(A_s)$ with residue field $\kappa$,
            the Hilbert function $\tilde{H'}$ of
            $(B/fB)\tensor_A \kappa$ is an element of $\HilbFunc$ and agrees
            with $H'$ for all arguments $\ee\in \II$, hence is equal to $H'$.
            For every $\ee\in \AA$,
            in the right-exact sequence
            \[
                \begin{tikzcd}
                    B_{\ee-\dd}\tensor_A \kappa\ar[r, "\mu_{f}\tensor_A
                    \kappa"] & B_{\ee}\tensor_A
                    \kappa \ar[r] & (B/fB)_{\ee}\tensor_A \kappa\ar[r] & 0
                \end{tikzcd}
            \]
            the dimensions of the $\kappa$-vector spaces are $H(\ee-\dd)$,
            $H(\ee)$,
            $H'(\ee)$. By~\eqref{eq:Hprim}, we have $H'(\ee) = H(\ee) -
            H(\ee-\dd)$ so
            this sequence is also left-exact. This proves that $u$ belongs to
            the $U'$ from the statement. Since $u$ was chosen arbitrarily, we have
            $\Spec(A_s)\subseteq U'$ which proves that $U'$ is open. The injectivity of $(\mu_f)_s$ and
            flatness of $(B/fB)_s$
            can be checked after localizing to every maximal
            ideal of a point $u\in \Spec(A_s)$ and then follow from the argument given
            above for $u_0$.
        \end{proof}

        \begin{theorem}\label{ref:opennessOfSaturation:thm}
            For any Hilbert function $H\colon \AA\to \mathbb{N}$,
            the set of points of $\Hilb^H$ which
            correspond to saturated ideals is open; we denote the resulting
            open subscheme by $\Hilb^{H, \sat}$. Moreover, for every
            morphism $\varphi\colon T\to \Hilb^H$ the following are equivalent
            \begin{enumerate}
                \item the (set-theoretic) image of $\varphi$ is contained in $\Hilb^{H,
                    \sat}$,
                \item the fibers of the universal ideal sheaf pulled back to
                    $T$ are saturated ideals.
            \end{enumerate}
            If these conditions holds, then the universal ideal sheaf pulled
            back to $T$ is saturated.
        \end{theorem}

        \begin{proof}
            \def\barL{\overline{L}}%
            By~\cite[Theorem~1.1]{Haiman_Sturmfels__multigraded}, the
            multigraded Hilbert scheme $\Hilb^H$ is quasi-projective over
            $\kk$, in particular it is locally of finite type over $\kk$. Since $\kk$ is Noetherian and of finite dimension, the
            scheme $\Hilb^H$ has an open cover by affine Noetherian schemes of
            finite dimension~\cite[Ex~6-7, p.~126]{Atiyah_Macdonald}.
            The universal family is, by definition, admissible with Hilbert
            function $H$, so
            Proposition~\ref{ref:transverseElementDeforms:prop} applies and
            yields openness.

            Once we proved openness, the claim about
            $\varphi\colon T\to \Hilb^H$ reduces immediately to the case of $T=\Spec(L)$, where $L$ is
            a field. Let $\kappa$ be the residue field of the point of
            $\Hilb^H$ given by $T\to \Hilb^H$, so that
            $\kappa\subseteq L$ is a field extension.
            Then the claim is
            that for a homogeneous ideal $I\subseteq S_\kappa$ and the ideal $I' = I
            \otimes_{\kappa} L\subseteq S_L$, obtained by a field
            extension, $I$ is saturated if and only if so is $I'$.

            Observe that $S_L/I'$ contains $S_{\kappa}/I$ and it is a free
            $S_{\kappa}/I$-module, since $L$ is a free $\kappa$-module.
            If $I'$ is saturated,
            then so is $I$ because of the containment. If $I$ is saturated, then a transverse element $\ell\in
            \Irr_{\kappa}$ is a nonzerodivisor on $S_{\kappa}/I$, hence also on the
            free module $S_{L}/I'$, so $I'$ is saturated.

            Finally, in order to show that the pulled back ideal sheaf is
            saturated, one may replace $T$
            by an open subset $U'\subseteq T$ as in Proposition~\ref{ref:transverseElementDeforms:prop} and apply the final 
            part of that Proposition.
        \end{proof}

        \newcommand{\Satbar}{\overline{\Sat}}%
        \begin{definition}
            We define the \emph{saturable locus} $\Satbar^H$ of $\Hilb^H$ as the closure
            of $\Hilb^{H, \sat}$. It is a union of irreducible components of
            $\Hilb^H$.
        \end{definition}

        As discussed in the introduction, it is interesting to
        understand the locus where a deformation of $I$ induces a deformation
        of $I^{\sat}$. To address this, we need to consider deformations of
        pairs $I\subseteq I^{\sat}$. The following subsection sets the stage
        for this (while~\S\ref{ssec:smoothnessOfProjections} contains the
        results).

    \subsection{Main diagram}\label{ssec:maindiagram}

        For (any field $\kappa$ and) ideals $K\subseteq J \subseteq \Sk$
        the two long exact sequences coming from $\Hom_{\Sk}(K, -)$ and
        $\Hom_{\Sk}(-,\frac{\Sk}{J})$ give the following \emph{main diagram}
        \[
            \begin{tikzcd}[column sep=small]
                &&&0\ar[d]  \\
                &&&\Hom_{\Sk}\left( K, \frac{J}{K} \right)_{\zerodeg}\ar[d]\\
                && \TgFlag\ar[r]\ar[d]\ar[dr, phantom, "\usebox\pullback" ,
                very near start, color=black] &\Hom_{\Sk}\left( K, \frac{\Sk}{K}
                \right)_{\zerodeg}\ar[d]\\
                0\ar[r] & \Hom_{\Sk}\left( \frac{J}{K}, \frac{\Sk}{J} \right)_{\zerodeg}\ar[r] &
                \Hom_{\Sk}\left( J, \frac{\Sk}{J} \right)_{\zerodeg} \ar[r] & \Hom_{\Sk} \left( K,
                \frac{\Sk}{J}
                \right)_{\zerodeg} \ar[r]\ar[d] & \Ext^1_{\Sk}\left( \frac{J}{K},
                \frac{\Sk}{J} \right)_{\zerodeg} \ar[rd, bend left=30]\\
                &&& \Ext^1_{\Sk}\left(K, \frac{J}{K}\right)_{\zerodeg}\ar[rd, bend
                right=30] & \ObFlag\ar[r]\ar[d]\arrow[dr, phantom,
                "\usebox\pullback" , very near start, color=black] &
                \Ext^1_{\Sk}\left( J,
                \frac{\Sk}{J} \right)_{\zerodeg}\ar[d]\\
                &&&& \Ext^1_{\Sk}\left(K, \frac{\Sk}{K}\right)_{\zerodeg}\ar[r] & \Ext^1_{\Sk}\left( K,
                \frac{\Sk}{J} \right)_{\zerodeg}.
            \end{tikzcd}
        \]
        We define the spaces $\TgFlag$ and $\ObFlag$ as pullbacks of vector
        spaces: their elements are pairs of elements which map to the same
        element. Their names are justified by the following theorem. Recall
        that the flag multigraded Hilbert scheme $\Hilb_{K\subseteq J}\to
        \Spec(\kk)$
        parameterizes pairs of ideals $K'\subseteq J'\subseteq \Sk$ with Hilbert functions
        $H_{\Sk/K'} = H_{\Sk/K}$ and $H_{\Sk/J'} = H_{\Sk/J}$, respectively, for every $\kk\to
        \kappa$.
        The fibre of $\Hilb_{K\subseteq J} \to \Spec(\kk)$ over the point
        $\Spec(\kappa)\to \Spec(\kk)$ is the Hilbert scheme parameterizing
        $K\subseteq J \subseteq \Sk$, where $\Sk$ is fixed. (If $\kk$ was a
        field and $\kappa = \kk$, then the fibre is equal to
        $\Hilb_{K\subseteq J}$.)
        The tangent space to this fiber at $[K\subseteq J]$ is $\TgFlag$, see, for
        example~\cite[Theorem~4.10]{Jelisiejew__Elementary}.
        The flag Hilbert scheme is
        a closed subscheme of $\Hilb_K \times \Hilb_J$, so we obtain
        projections
        \begin{equation}\label{eq:Hilbs}
            \begin{tikzcd}
                \Hilb_K & \Hilb_{K\subseteq J} \ar[r, "\pr_J"]\ar[l, "\pr_K"'] & \Hilb_J.
            \end{tikzcd}
        \end{equation}
        Let
        \[
            \psi\colon \Hom_{\Sk}(K,\Sk/K) \oplus \Hom_{\Sk}(J, \Sk/J)\to
            \Hom_{\Sk}(K, \Sk/J)
        \]
        be the sum of the natural maps $\Hom_{\Sk}(K, \Sk/K), \Hom_{\Sk}(J, \Sk/J)\to
        \Hom_{\Sk}(K, \Sk/J)$. We prove that if $\psi_{\zerodeg}$ is surjective,
        then $\Hilb_{K\subseteq J}\to \Spec(\kk)$ admits an obstruction theory at
        $[K\subseteq J]$ with obstruction space $\ObFlag$, see
        Theorem~\ref{ref:obstructionAtFlag:thm}. Since the proof is
        technical, we delegate it to an appendix.
        For some smoothness results, we will need
        the full tangent space to
        $[K\subseteq J]\in \Hilb_{K\subseteq J}$, that is, also its image in
        $\Spec(\kk)$. We can control it
        thanks to Lemma~\ref{ref:sernesi:lem}.

        Going back to the main diagram, we see that the ``linking'' group
        $\Ext^1_{\Sk}(K,
        J/K)_{\zerodeg}$ is the obstruction group for equivariantly deforming
        $K$ inside $J$, in other words, it is the obstruction group of the
        fiber
        $(\pr_J^{-1}([J]), [K])$, while $\Ext^1_{\Sk}(J/K, \Sk/J)_{\zerodeg}$ is the obstruction
        group for deforming $J/K$ in $\Sk/K$, hence the obstruction group of
        $(\pr_K^{-1}([K]), [J])$. We will be principally interested in the map
        $\pr_K$, hence we define the \emph{fiber obstruction group}
        \begin{equation}\label{eq:obfib}\tag{ObFib}
            \ObFib(K, J) := \Ext^1_{\Sk}\left( \frac{J}{K}, \frac{\Sk}{J}
            \right)_{\zerodeg}.
        \end{equation}
        \stepcounter{equation}

        The Main Theorem of obstruction calculus
        (see~Lemma~\ref{ref:obstructionCalculus:lem}) asserts that a map of pointed schemes
        which induces a surjection of tangent spaces and an injection of
        obstruction spaces (at the base point) is smooth (at the base point).
        It follows from the main diagram (and Lemma~\ref{ref:sernesi:lem}) that the vanishing of $\ObFib$
        implies \emph{both} these conditions for $\pr_K$ (when $\psi_{\zerodeg}$ is onto).
        We will exploit this observation
        in~\S\ref{ssec:smoothnessOfProjections}.

    \subsection{Toric varieties setup}\label{ref:toric_setup:subsec}

    In this subsection we restrict the general setup from~\S\ref{ssec:setup} to a
    more geometric one, namely that of toric varieties. Our general reference
    for toric varieties are~\cite{cox_book, fulton_toric}.
    Also~\cite{Maclagan_Smith__regularity, Maclagan_Smith__uniform}
    are very helpful; in this subsection we do little more
    than applying them. Finally, the reader uninterested in the general toric
    case may go straight into Example~\ref{ex:productOfProjectives}.

    Let us formally introduce our toric setup.
    Fix a smooth complete fan $\Sigma$ corresponding to a \emph{projective} toric
    variety. We take $S$ to be a polynomial ring over $\kk$ with variables
    given by the rays of $\Sigma$. The irrelevant ideal $\Irr = \Irr(\Sigma)$ is the monomial ideal
    generated by elements $\prod\{x_i\ |\ i\not\in \sigma\}$ where $\sigma$ ranges over the maximal
    cones of $\Sigma$, see~\cite[p.207]{cox_book}. The grading is given by the
    torsion-free abelian group $\AA := \Pic(\Sigma)$, where $\Pic(\Sigma)$ is defined as
    in~\cite[Theorem~4.2.1]{cox_book}. To further distinguish this setup from
    the general one, we will speak about $\Pic(\Sigma)$-homogeneous ideals,
    rather than abbreviate it to homogeneous ideals.

    Arguing as in~\cite{cox_book}, but relatively over $\kk$,
    we obtain a smooth projective morphism $X \to \Spec(\kk)$ with toric
    fibers. All the fibers have the same Nef cone in $\AA$, to which we refer
    simply as $\Nef(X)$. Similarly, we refer to $\Pic(X)$ and $\Irr(X)$,
    rather than $\Pic(\Sigma)$, $\Irr(\Sigma)$.

    \begin{example}\label{ex:productOfProjectives}
        Take $\Sigma$ to be the fan of the product of projective spaces
        $\mathbb{P}^{n_1}\times \cdots \times \mathbb{P}^{n_s}$. In this case,
        the polynomial ring $S$ is
        \[
            S = \kk[\alpha_0, \ldots ,\alpha_{n_1},\beta_0, \ldots
            ,\beta_{n_2}, \gamma_0,  \ldots , \gamma_{n_3},\ldots],
        \]
        the grading is given by $\deg(\alpha_{\bullet}) = (1, 0,0, \ldots
        ,0)$, $\deg(\beta_{\bullet}) = (0,1,0, \ldots )$ and so on and the
        irrelevant ideal is
        \[
            \Irr\left( \mathbb{P}^{n_1}\times \cdots \times
            \mathbb{P}^{n_s} \right) = \left( \alpha_0, \ldots , \alpha_{n_1}\right)\cdot \left(
            \beta_0, \ldots ,\beta_{n_2} \right)\cdot  \ldots.
        \]
        This example is the most important in terms of applications to border
        apolarity. Even the case $s = 1$ is useful.
    \end{example}

    We now ask how far is a $\Pic(X)$-homogeneous ideal $I\subseteq \Sk$ from its saturation. In general, this
    is a subtle question, see~\cite{Maclagan_Smith__regularity,
    Maclagan_Smith__uniform}.
    For any ideal $I$, we have the following result.

        \newcommand{\nn}{\mathbf{n}}%
		\newcommand{\uu}{\mathbf{u}}%

	\begin{lemma}\label{lem:ideal_agrees_with_sat}
	If $I$ is a $\Pic(X)$-homogeneous ideal of $\Sk$, then there exists $\nn\in \Nef(X)$ such that
	$I_\aa = I^{\sat}_\aa$ for every $\aa \in  \nn + \Nef(X)$.
	\end{lemma}	
	\begin{proof}
	Let $M  =I^{\sat}/I$. By \cite[Corollary~3.8]{Maclagan_Smith__regularity}, there exists $\nn\in \Nef(X)$ such that $H^i_B(M)_{\nn+\uu} = 0$ for all $		\uu\in \Nef(X)$. In particular, $H^0_B(M)_{\nn+\uu} = 0$ which by the definitions of zeroth local cohomology and saturation implies that $M_{\nn+\uu} = 0$.
	\end{proof}

        \begin{proposition}\label{ref:projectionFromSatLocallyClosed:prop}
            Let $I$ be a homogeneous ideal of $\Sk$ and $I^{\sat}$ be its
            saturation. Let
            \[
                p := \pr_I\colon \Hilb_{I\subseteq I^{\sat}}\to \Hilb_I,
            \] then $p$ is a locally closed immersion near $[I\subseteq
            I^{\sat}]$.
            More precisely, let $U$ be the open neighborhood of $[I\subseteq
            I^{\sat}]\in \Hilb_{I\subseteq I^{\sat}}$ that consists of pairs
            $[K\subseteq J]$ with $J$ saturated and let $Z = \Hilb_{I\subseteq
            I^{\sat}} \setminus\; U$.
            Then $p|_U\colon U\to \Hilb_I\setminus\; p(Z)$ is a closed
            immersion.
        \end{proposition}
        \begin{proof}
            First, the subset $U$ is equal to
            $\pr_{I^{\sat}}^{-1}(\Hilb_{I^{\sat}}^{\sat})$, hence indeed open
            by Theorem~\ref{ref:opennessOfSaturation:thm}.
            Observe that
            $p^{-1}(p(Z)) = Z$ as closed subsets.
            Suppose not, then there are
            points $[I'\subseteq J']\in U$ and
            $[I' \subseteq J]\in Z$, that is $J'$ is saturated and $J$ is not. By Lemma~\ref{lem:ideal_agrees_with_sat}
            the Hilbert functions of $I$ and $I^{\sat}$ are equal in all multidegrees far enough into the interior of the Nef cone.
            Therefore, so are the Hilbert functions of $I'$, $J$ and $J'$. It
            follows that these three ideals define the same scheme in
            $X$. As a result, we have $(I')^{\sat} = J^{\sat} =
            (J')^{\sat} = J'$.
            The Hilbert functions of $J$ and $J' =
            J^{\sat}$ are equal, so the inclusion $J\subseteq J^{\sat} = J'$ is an
            equality, hence $J = J'$.
            The proof of $p^{-1}(p(Z)) = Z$ is completed.

            The morphism $p$ is
            projective~\cite[Corollary~1.2]{Haiman_Sturmfels__multigraded} and
            so restricts to a \emph{projective}
            morphism
        \[
            p|_U\colon U = \Hilb_{I\subseteq I^{\sat}}\setminus\; p^{-1}(p(Z))
            \to \Hilb_I \setminus\; p(Z).
        \]
        We claim that this morphism is a closed immersion. By~
        \cite[\href{https://stacks.math.columbia.edu/tag/04XV}{Tag 04XV},
        \href{https://stacks.math.columbia.edu/tag/05VH}{Tag
        05VH}]{stacks-project} it is enough to prove that for every point
        $[I']\in \Hilb_I \setminus\; p(Z)$ the fiber
        $p^{-1}([I'])$
        is either empty or maps isomorphically to the point $[I']$.
		Let $\kappa_0$ be the residue field at $[I']$ and $\kappa$ be its algebraic closure.
        It is enough to prove the above claim about the fiber after field
        extension $\kappa_0 \to \kappa$~\cite[\href{https://stacks.math.columbia.edu/tag/02L4}{Tag
        02L4}]{stacks-project}. In particular, we may assume that 
        the base field is algebraically closed and that $[I']$ is closed.
        Consider the
        fiber $p^{-1}([I'])$.
        Arguing as above, we get that \emph{as a set} the fiber is
        \begin{itemize}
            \item a singleton $[I'\subseteq (I')^{\sat}]$ if
                $H_{\Sk/(I')^{\sat}} = H_{\Sk/I^{\sat}}$,
            \item empty if $H_{\Sk/(I')^{\sat}} \neq H_{\Sk/I^{\sat}}$.
        \end{itemize}
        Assume that the fiber is nonempty.
        Since $\Hom_{\Sk}((I')^{\sat}/I', \Sk/(I')^{\sat}) = 0$, by the main
        diagram~\ref{ssec:maindiagram} we get that the tangent map to $p$
        at $[I'\subseteq (I')^{\sat}]$ is injective. Hence the fiber has zero
        tangent space, so \emph{as a scheme} it is a point.
        \end{proof}

        \subsubsection{Relation to Maclagan-Smith's and Grothendieck's Hilbert
        schemes}\label{ssec:MaclaganSmithHilbertScheme}
        We now relate the above to a more geometric setup related to $X$.

        Let $P$ be a polynomial in rank $\Pic(X)$ variables.
        Maclagan and Smith constructed the Hilbert scheme $\Hilb_X^P$
        that parameterizes subschemes of $X$ with multigraded Hilbert
        polynomial equal to $P$, see~\cite[Thm.~6.2]{Maclagan_Smith__uniform}. When
        $X$ is a projective space, this agrees with Grothendieck's Hilbert scheme.
        We recall the construction,
        following~\cite[Theorem~6.2]{Maclagan_Smith__uniform}.

        Fix a $\Pic(X)$-graded polynomial $P$.
        By~\cite[Theorem~4.11]{Maclagan_Smith__uniform} there exists an
        element $\nn\in \Nef(X)$ such that for every $\aa\in \nn + \Nef(X)$ and for
        every subscheme
        $Z$ in a fiber $X_{\kappa}$ with multigraded Hilbert polynomial equal to $P$, we have
        \begin{equation}\label{eq:constantRank}
            \dim_{\kappa}(\Sk/I_Z)_{\aa} = P(\aa),
        \end{equation}
        where $I_Z$ is the saturated ideal of $Z$.

        Fix such an $\nn$ and consider a
        function $h\colon \Pic(X) \to \mathbb{N}$ such that $h(\aa) = P(\aa)$
        for all $\aa \in \nn + \Nef(X)$ and $h(\aa) = \operatorname{rank}_\kk S_\aa$ otherwise. Consider
        Haiman-Sturmfels multigraded Hilbert scheme for function $h$.
        Maclagan-Smith~\cite[Theorem~6.2]{Maclagan_Smith__uniform} prove that
        for any $\nn$ as above, the resulting scheme represents $\Hilb_X^P$.
        (There is a subtlety in this construction, related to the fact that
        there exist indices $\bb$ such that $S_{\bb}$ is nonzero and yet
        $\bb+\Nef(X)\not\subset \Nef(X)$. It will not be important for us, so
        we omit discussing it and refer to~\cite{Maclagan_Smith__uniform} for
        details.)

        For every Hilbert function $H\colon \Pic(X)\to \mathbb{N}$ such that
        there exists an associated multigraded Hilbert polynomial $P$, we
        have a morphism $\Hilb^H \to \Hilb^P(X)$,
        which is given by forgetting the degrees outside $\nn + \Nef(X)$.

        \begin{corollary}\label{ref:fixedHilbertFunctionLocallyClosed:cor}
            For every Hilbert function $H\colon \Pic(X)\to \mathbb{N}$ such that
            there exists an associated multigraded Hilbert polynomial $P$, the
            morphism
            \[
                p\colon \Hilb^{H, \sat}\to \Hilb^P(X)
            \]
            is a locally closed immersion.
        \end{corollary}
        \begin{proof}
            Fix an $\nn$ as above.  Take a point $[I']\in \Hilb^{H, \sat}$.
            Let $I := I'|_{\nn + \Nef(X)}$. The ideals $I'$ and $I$ define the
            same subscheme of $X$ and $I'$ is saturated, so $I' = I^{\sat}$.
            Maclagan-Smith's result implies that $\Hilb^P(X)$ is isomorphic to
            $\Hilb_I$.  The scheme $\Hilb^{H, \sat}$ is isomorphic to $U$ from
            Proposition~\ref{ref:projectionFromSatLocallyClosed:prop}. Thus
            the result follows from this Proposition.
        \end{proof}

    \subsection{Macaulay's inverse
    systems}\label{ssec:macaulaysInverseSystems}

    Throughout the paper, we will frequently consider surjections such as
    $\Sk/I\onto \Sk/I^{\sat}$ for a field $\kappa$ with $\kk \to
    \kappa$. We will also analyse the possible choices of the subideal $I$ for fixed
    $I^{\sat}$. It seems cleaner to do this dually: to consider superobjects
    $I^{\perp}\supseteq (I^{\sat})^{\perp}$ instead. Hence we employ
    Macaulay's inverse systems. They are a standard tool in the Artinian case and
    introduced in the higher-dimensional case at least
    in~\cite{Elias_Rossi__Gorenstein_for_higher_dims}. Here we loosely
    summarize exactly as much of the theory as we need.

    The idea is quite simple yet requires some notation. Given $\Sk =
    \kappa[\alpha_0,  \ldots ,\alpha_n]$ we
    define the dual graded vector space $\Sk^* = \kappadp[x_0, \ldots ,x_n]$ and
    make it into a graded $\Sk$-module by the \emph{contraction} action
    $\Sk\hook
    \Sk^*\to \Sk^*$ defined on a monomial $m\in \Sk^*$ by
    \[
        \alpha_i \hook m =
        \begin{cases}
            0 & \mbox{if } x_i\mbox{ does not divide } m,\\
            m/x_i & \mbox{otherwise}.
        \end{cases}
    \]
    For an ideal $I\subseteq \Sk$ we define $I^{\perp}\subseteq
    \Sk^*$ as $I^{\perp} =
    \left\{ f\in \Sk^*\ |\ I\hook f = 0 \right\}$. The subspace $I^{\perp}$ is
    an $\Sk$-submodule of $\Sk^*$. Similarly, for a submodule
    $M\subseteq \Sk^*$ we
    define an ideal $M^{\perp} \subseteq \Sk$ as the annihilator of $M$, namely
    $M^{\perp} = \left\{ \sigma\in \Sk\ |\ \sigma\hook M = 0 \right\}$.
    We have $(-)^{\perp\perp} = \id$, hence $(-)^{\perp}$ gives a bijection
    between graded ideals of $\Sk$ and graded $\Sk$-submodules of $\Sk^*$. This
    bijection reverses inclusions and is equivariant with respect to the usual
    $\GL_{n+1}$-actions. For an element $F\in \Sk^*$ we also write $F^{\perp}$
    as a shorthand for $(\Sk\hook F)^{\perp}$.

	\begin{remark}\label{ref:contraction_and_differentiation:rmk}
        For experts: the subscript in $\kappadp$ underlies the fact that $\Sk^*$ can
        be made into a divided powers ring on which $(\Sk)_1$ acts as differential
        operators (we will not use this structure).
    \end{remark}

    In characteristic zero (or large enough) a slightly different,
    equivalent, action by \emph{partial differentiation}
    is usually employed.  In this setup, we assume that $\kappa$ has
    characteristic zero, take the polynomial ring $\kappa[x_0, \ldots ,x_n]$ (note the lack of
    subscript $dp$) and make it into an $\Sk$ module where $\alpha_i$ acts as
    $\frac{\partial}{\partial x_i}$. We denote this action by $\circ$.
    There is an isomorphism $\kappa[x_0, \ldots ,x_n]\to \Sk^*$ of $\Sk$-modules that
    sends a monomial $x^{\aa}$ to $\aa!x^{\aa}$, and so all the duality
    claims for $\Sk, \Sk^*$ hold --- still in characteristic zero --- also for
    the pair $\Sk, \kappa[x_0, \ldots ,x_n]$. See~\cite[Appendix~A]{iakanev} or \cite{Jel_classifying} for more details.

	We use the following example in \S\ref{ssec:five} where we assume 
	that $\kappa$ has characteristic zero.
	Therefore, we present it under the same assumption and we consider the action 
	$\circ$ of $\Sk$ on $\kappa[x_0,\ldots, x_n]$.

    \begin{example}\label{ex:points}
		Assume that $\kappa$ has characteristic zero.         
        Let $L\in \spann{x_0, \ldots ,x_n}$ be a nonzero linear form and let $I(\{L\})$ be the homogeneous ideal of the point $[L]\in
        \Proj \kappa[\alpha_0, \ldots ,\alpha_n] \simeq  \mathbb{P}^{n}$. For every non-negative integer $i$ we have 
        \[
            I(\{L\})^\perp_i = \langle L^i \rangle.
        \]
        Similarly, if $L_1, \ldots ,L_r$ are pairwise different linear forms,
        then the equality 
        $I(\{L_1, \ldots ,L_r\})^\perp_i = \langle  L^i_j\ |\  1\leq j\leq
        r\rangle$ holds for every non-negative integer $i$.
    \end{example}

    For further use, we note the following result:
    \begin{proposition}\label{ref:HilbertFunctionFromLimits:prop}
        Let $F_1, \ldots ,F_s$ be homogeneous degree $\ee$ elements of $\Sk[t]$, where
        $\deg(t) = \zerodeg$. Assume that there are homogeneous degree $\ee$ elements $G_1$, \ldots
        ,$G_s\in \Sk[t]$
        such that
        \begin{enumerate}
            \item for every $i$ there exists $d_i\geq 0$ with $t^{d_i}G_i\in
                \spann{F_1, \ldots ,F_s}\Sk[t]$,
            \item\label{it:indep} $G_1|_{t=0}$, \ldots ,$G_{s}|_{t=0}$ are linearly
                independent. (we call them \emph{limit forms}).
        \end{enumerate}
        Then $F_1|_{t=\lambda}$, \ldots ,$F_{s}|_{t=\lambda}$ are linearly
        independent for general $\lambda$ and the limit of the subspaces
        $\spann{F_1, \ldots ,F_s}$ at $t\to 0$ is $\spann{G_1|_{t=0}, \ldots ,G_{s}|_{t=0}}$.
    \end{proposition}
    \begin{proof}
        Let $N\subseteq (\Sk)_{\ee}$ be a finite-dimensional
        $\kappa$-subspace spanned by all
        monomials in $F_1$, \ldots, $F_s$, $G_1$, \ldots, $G_s$.
        Let $M := \kappa[t]^{\oplus s}$ and let $\iota\colon M\to N[t]$ be the
        $\kappa[t]$-linear map that sends the $i$-th generator to $G_i$.
		Since $t$ is a nonzerodivisor on $N[t]$, it follows from~\ref{it:indep}
		that $\iota$ is injective and that $t$ is a nonzerodivisor in the
        $\kappa[t]$-module $N[t]/\im \iota$.
        By the classification of finitely generated $\kappa[t]$-modules there
        exists an $f\in \kappa[t]\setminus (t)$ such that after inverting $f$
        the homomorphism $\iota$ is
        injective and $(N[t]/\im \iota)_{f}$ is a free
        $\kappa[t]_{f}$ module.
        By this freeness it follows that $\iota|_{t=\lambda}$ is injective
        for every $\lambda$ with $f(\lambda)\neq 0$.
        This shows that for every nonzero such $\lambda$, the forms
        $(G_j|_{t=\lambda})_{j=1}^s$ are linearly independent. They are linear
        combinations of $(F_j|_{t=\lambda})_{j=1}^s$, which are thus linearly
        independent as well. The last statement is purely formal.
%
    \end{proof}

\subsection{Miscellaneous results}
In this subsection we gather two rather technical yet general results, which form
important parts of the arguments in~\S\ref{sec:examples}
and~\S\ref{ssec:smoothnessOfSaturated}, respectively.
\begin{lemma}[small tangent space implies large
        intersection, {\cite[Lemma~2.6]{Mandziuk}}]\label{ref:boundOnIntersectionFromTangentSpace:lem}
    Let $X$ be a scheme locally of finite type over a field $\kappa$ and $Z_1,
    Z_2\subseteq X$ be irreducible closed subschemes. Let $x\in Z_1\cap Z_2$
    be a $\kappa$-point. Then every irreducible component of $Z_1\cap Z_2$ has
    dimension at least
    \[
        \dim Z_1 + \dim Z_2 - \dim T_x X.
    \]
\end{lemma}

We will also need the following result about derivations.
Let $X$ be a $\kk$-scheme and $\varphi\colon \Spec(A)\to X$ be a morphism of
schemes. The $A$-module
\[
    \Hom_A(H^0(\varphi^*\Omega_{X/\kk}), A)
\]
is called
the \emph{module of (pulled back) vector fields on $\Spec(A)$}. We have the
following description.
\begin{lemma}[vector fields and cotangent module]\label{ref:vectorfields:lem}
    There is a bijection functorial in $A$ between the elements of the module
    of vector fields on $\Spec(A)$ and the set of morphisms of $\kk$-schemes
    $\varphi'\colon \Spec(A[\varepsilon]/\varepsilon^2)\to X$ which are equal
    to $\varphi$
    when restricted to $\Spec(A)$.
\end{lemma}
The proof is standard, yet we could not find a ready reference. Likely the
most natural way of arguing is gluing on affine parts, but this seems messy to
write down. As a result, the proof is quite hermetic, the reader might like to
follow it for affine $X$ first.
\begin{proof}
    \def\cA{\mathcal{A}}
    Let $\cA := \varphi_*\OO_{\Spec(A)}$. To give a morphism $\varphi'$ as in the statement is
    to give a sheaf-of-$\kk$-algebras homomorphism $(\varphi')^{\#}\colon \OO_X\to
    \cA[\varepsilon]/\varepsilon^2$ that restricts to $\varphi^{\#}$ when
    composed with $\cA[\varepsilon]/\varepsilon^2\onto \cA$. Write $(\varphi')^{\#}
    = \varphi^{\#} + \varepsilon \delta$ for $\delta\colon \OO_X\to \cA$.
    The fact that $(\varphi')^{\#}$ is a homomorphism is equivalent to the
    fact that $\delta$ is a $\kk$-linear derivation of the $\OO_X$-module
    $\cA$. By the universal property of $\Omega_{X/\kk}$, such derivations are
    in bijection with $\Hom_{\OO_X}(\Omega_{X/\kk}, \cA)$, which by adjunction
    is isomorphic $\Hom_{\OO_{\Spec(A)}}(\varphi^* \Omega_{X/\kk},
    \OO_A)$, which is in bijection with the module of vector fields on
    $\Spec(A)$.
\end{proof}

        \section{Main results}

        We work in the toric setup as in
        Subsection~\ref{ref:toric_setup:subsec}. In particular, $X$ is a smooth
        projective toric variety over a Noetherian ring $\Bbbk$.

        We start by giving cohomological criteria for the existence of sticky
        ideals (see introduction). We then study the locus of saturated ideals
        with a given Hilbert function inside the usual Hilbert scheme of points of $X$

        After that we restrict the setup to the projective space
    where we are able to do two more things: first, express the obstruction groups from Subsection~\ref{ssec:smoothnessOfProjections}
    in a more explicit way. Second, in Subsection~\ref{ref:smoothness_of_saturated:sec} we provide classes of
    saturated ideals that give smooth points of the corresponding multigraded
    Hilbert scheme.

    \subsection{Smoothness of projection maps}\label{ssec:smoothnessOfProjections}

    Recall that $S$ is the Cox ring of a smooth projective toric variety $X$ over a Noetherian ring $\Bbbk$
    and $I,J,K$ are $\Pic(X)$-homogeneous ideals of $\Sk$. If $K\subseteq J$ then recall that
    $\ObFib(K,J) = \Ext^1_{\Sk}(J/K, \Sk/J)_\zerodeg$.

        \begin{theorem}\label{ref:deformsWhenObFibVanishes:thm}
            Suppose that we have $\ObFib(K, J) = 0$. Then
            $\pr_K\colon \Hilb_{K\subseteq J}\to \Hilb_K$ is smooth at $[K\subseteq J]$.
            In particular, if $K\subsetneq J\subseteq K^{\sat}$ then
            $K$ is nonsaturable.
        \end{theorem}
        \begin{proof}
            Since $\ObFib = 0$, the assumptions of
            Theorem~\ref{ref:obstructionAtFlag:thm} are satisfied.
            By the main diagram~\S\ref{ssec:maindiagram} the obstruction map
            is injective and the tangent map to $\pr_K|_{\kappa}\colon \Hilb_{K\subseteq
                J}|_\kappa\to \Hilb_K|_\kappa$ is surjective. By Lemma~\ref{ref:sernesi:lem} and the five lemma this implies that 
                the tangent map to $\pr_K\colon \Hilb_{K\subseteq
                J}\to \Hilb_K$ is also surjective. This proves that $\pr_K$
                is smooth at $[K\subseteq J]$. In
                particular its image contains an open neighborhood of $[K]$.
                Suppose that $K$ is saturable and let $K_t$ be a family with
                special fiber $K$ and general fiber $K_{\eta}$ such that
                $K_{\eta}$ is saturated. Then, after perhaps shrinking the
                base of the family, we would also get a family $K_t\subseteq
                J_t$ with special fiber $J$.
                By Lemma~\ref{lem:ideal_agrees_with_sat} there exists $\nn\in \Nef(X)$ such that
                $K_\aa = K^{\sat}_\aa$ for every $\aa\in \nn+\Nef(X)$. As a result, $(J/K)_\aa = 0$
                for such $\aa$ and the same holds for $(J_{\eta}/K_{\eta})_\aa$. It follows that 
                $J_{\eta}^{\sat} = K_\eta^{\sat}$.
				From this we conclude that $K_\eta \subsetneq J_\eta \subseteq
				J_\eta^{\sat} = K_{\eta}^{\sat}$
				which contradicts the assumption that $K_\eta$ is saturated.				
%
        \end{proof}

        It may happen that the obstruction group is nonzero, yet the map
        $\pr_K$ is smooth, so we provide the following refined criterion.
        \begin{theorem}\label{ref:deformWhenTangentSurjective:thm}
            Suppose that $(K,J)$ are such that the map $\Hom_{\Sk}(J,
            \Sk/J)_{\zerodeg}\to \Hom_{\Sk}(K, \Sk/J)_{\zerodeg}$ is surjective
            and that $[J]\in \Hilb_J$ is a smooth point. Then $\pr_K$ is
            smooth at $[K\subseteq J]$. 
            In particular, if $K\subsetneq J\subseteq K^{\sat}$ then
        $K$ is nonsaturable.
        \end{theorem}
        \begin{proof}
            By assumption, the map $\psi_{\zerodeg}$ is surjective, so we can apply
            Theorem~\ref{ref:obstructionAtFlag:thm}.
            By pullback, the map $d\pr_K|_{\kappa}\colon \TgFlag\to
            \Hom_{\Sk}(K,\Sk/K)_{\zerodeg}$ is surjective.
            The map $\ObFlag\to
            \Ext^1_{\Sk}(J, \Sk/J)_{\zerodeg}$ is a map of obstruction theories,
            so the image of any obstruction is the obstruction to deforming
            $[J]$, which is zero by assumption. Thus the obstructions actually
            live in a smaller space $O := \ker(\ObFlag\to \Ext^1_{\Sk}(J,
            \Sk/J)_{\zerodeg})$ which is thus by definition another obstruction space at
            $[K\subseteq J]\in \Hilb_{K\subseteq J}$. Since $O\to \Ext^1_{\Sk}(K,
            \Sk/K)_{\zerodeg}$ is injective, by Lemma~\ref{ref:sernesi:lem} we
            obtain surjectivity on tangent spaces and then it follows from the Main Theorem of
            obstruction calculus that $\pr_K$ is smooth at $[K\subseteq J]$.
        \end{proof}


        In the special case $(K, J) = (I, I^{\sat})$ we can be a little more
        precise.
        \begin{theorem}\label{ref:saturationDeformsWhenObFibVanishes:thm}
            Let $I$ be a homogeneous ideal and
            suppose that $\ObFib(I, I^{\sat}) = 0$. Then $\pr_{I}$ is an open immersion
            near $[I \subseteq I^{\sat}]$. In particular, if $I$ is non
            saturated, it is entirely nonsaturable.
        \end{theorem}
        \begin{proof}
            Since $\ObFib = 0$, the assumptions of
            Theorem~\ref{ref:obstructionAtFlag:thm} are satisfied. 
            Let $\ell$ be a transverse element as
            in~Section~\ref{ssec:saturations}.
 			 The multiplication by $\ell$ is an
            injective map on $\Sk/I^{\sat}$ and a nilpotent one on
            $I^{\sat}/I$, hence $\Hom_{\Sk}(I^{\sat}/I, \Sk/I^{\sat}) = 0$.
            By Lemma~\ref{ref:sernesi:lem} and the main diagram, 
            the tangent map to $\pr_I\colon \Hilb_{I\subseteq
                I^{\sat}}\to \Hilb_I$ is bijective, while the obstruction map
                is injective. This proves that $\pr_I$ is \'etale at $[I\subseteq I^{\sat}]$. In
                particular its image contains an open neighborhood of $[I]$,
                whence if $I^{\sat}\neq I$, then $I$ is entirely nonsaturable.
                Consider  $U := \pr^{-1}_{I^{\sat}}(\Hilb_{I^{\sat}}^{\sat})$
                which is an open
                (Theorem~\ref{ref:opennessOfSaturation:thm}) neighbourhood of
                $[I\subseteq I^{\sat}]$. For every point $[I'
                \subseteq I'']\in U$ we have $I'' = (I')^{\sat}$. Thus the map
                $(\pr_I)|_U\colon U\to \Hilb_I$ is universally
                injective~\cite[\href{https://stacks.math.columbia.edu/tag/01S4}{Tag
                01S4}]{stacks-project}. By shrinking $U$ to a neighborhood of
                $[I\subseteq I^{\sat}]$, we assume that it is \'etale on $U$
                as well. By~\cite[\href{https://stacks.math.columbia.edu/tag/025G}{Tag
                025G}]{stacks-project} the morphism $(\pr_I)|_U$ is an open
                immersion.
        \end{proof}

        Again, the vanishing of $\ObFib(I, I^{\sat})$ is not
        necessary for $I$ to be nonsaturable, see~\S\ref{ssec:nonvanishingGroup}
        for an explicit example. We state a slightly more
        general version of the theorem.
        \begin{theorem}\label{ref:saturationDeformWhenTangentSurjective:thm}
            Suppose that the map $\Hom_{\Sk}(I^{\sat},
            \Sk/I^{\sat})_{\zerodeg}\to\Hom_{\Sk}(I, \Sk/I^{\sat})_{\zerodeg}$ is
            surjective and that $[I^{\sat}]\in \Hilb_{I^{\sat}}$
            is a smooth point. Then $\pr_I$ is an open immersion near $[I\subseteq
            I^{\sat}]$. In particular, if $I$ is not saturated, it is entirely
            nonsaturable.
        \end{theorem}
        \begin{proof}
            By Theorem~\ref{ref:deformWhenTangentSurjective:thm} the map is
            smooth and arguing as in the proof of
            Theorem~\ref{ref:saturationDeformsWhenObFibVanishes:thm} we see
            that $\pr_I$ is an open immersion near $[I\subseteq I^{\sat}]$.
        \end{proof}
        
        \begin{example}\label{exa:Z2_grading}
		    Let $\kappa$ be a field of characteristic zero and $X=\mathbb{P}^1\times \mathbb{P}^1$
		    with Cox ring $\Sk=\kappa[\alpha_0,\alpha_1,\beta_0,\beta_1]$ 
		    where $\deg(\alpha_i) = (1,0)$ and $\deg(\beta_i) = (0,1)$. The irrelevant ideal
		    is $\Irr = (\alpha_0,\alpha_1)\cdot (\beta_0,\beta_1)$. If 
		    $K=(\alpha_0\alpha_1,\alpha_0\beta_0,\alpha_0\beta_1,\beta_0\beta_1)$
		    and $J = K^{\sat} = (\alpha_0, \beta_0\beta_1)$, then
            $\Ext^1_{\Sk}(J/K, \Sk/J)_\zerodeg = 0$.
		    Therefore, $K$ is entirely nonsaturable by 
		    Theorem~\ref{ref:saturationDeformsWhenObFibVanishes:thm}.
        \end{example}
        
		The second example illustrates that choosing $J$ different than $K^{\sat}$ may be useful.        
        
        \begin{example}\label{exa:intermediate_ideal}
		Let $\kappa$ be a field of characteristic zero, $X=\mathbb{P}^2$ with Cox ring $\Sk=\kappa[\alpha_0,\alpha_1,\alpha_2]$
		and consider the ideal $K = (\alpha_0^2\alpha_2,\alpha_0\alpha_1^3,\alpha_0^2\alpha_1^2,\alpha_0^3\alpha_1,
		\alpha_0^5,\alpha_1^6)$. We show that $K$ is nonsaturable. We have
        $\Ext^1_{\Sk}(K^{\sat}/K, \Sk/K^{\sat})_\zerodeg \neq 0$
		so Theorem~\ref{ref:saturationDeformsWhenObFibVanishes:thm} cannot be
        used. Yet,
        for $J = K^{\sat}_{\geq 3}$ we get $\Ext^1_{\Sk}(J/K, \Sk/J)_\zerodeg
        = 0$. Therefore, $K$ is nonsaturable by
        Theorem~\ref{ref:deformsWhenObFibVanishes:thm}.
        \end{example}

		For applications for border rank lower bounds the following variation of the above ideas
		may be useful. We present two versions: first for an arbitrary smooth projective toric variety
		and then for the product of projective spaces, since the latter is a main case of interest for applications.
        The general statement is quite heavy. The informal idea is that $K$ is
        a truncation of $I$, so by replacing $I$ with $K$ we reduce the amount
        of data.

		\begin{proposition}\label{ref:def_criterion_general_toric}
            Assume that the ideal
			$\Irr$ is minimally generated in degrees $\vv_1, \ldots, \vv_l$.
			Let $I\subseteq \Sk$ be a $\Pic(X)$-homogeneous ideal and $A\subseteq \Pic(X)$ be a subset
			such that $A+\Eff(X) \subseteq A$. Let $L$ denote the ideal 
			$\bigoplus_{\aa\in A}(\Sk)_\aa$ and $K=I+L$. If there exist a $\Pic(X)$-graded
			ideal $J\subseteq \Sk$, a degree $\uu\in \Eff(X)$, and positive integers $k_1,\ldots, k_l$ such that
			\begin{enumerate}
				\item \label{it:1} $K \subseteq J$,		
				\item \label{it:4} $\ObFib(K, J) = 0$,
				\item\label{it:2} $K_\uu \subsetneq J_\uu$,
				\item\label{it:3} $K_{\uu+k_i\cdot \vv_i} = J_{\uu + k_i \cdot
                    \vv_i} \subsetneq (\Sk)_{\uu+k_i\cdot \vv_i}$
				 for all $1\leq i \leq l$,
			\end{enumerate}			 
			then $I$ is not saturable.
		\end{proposition}
		\begin{proof}
		By part~\ref{it:1} we may consider the flag multigraded Hilbert scheme 
		$\Hilb_{K\subseteq J}$. As in the proof of 
		Theorem~\ref{ref:deformsWhenObFibVanishes:thm},
		it follows from part~\ref{it:4} that there there is an open neighborhood
		$U$ of $K$ in $\Hilb_K$ such that for every $K' \in U$ there is $J'$ such that
		$K'\subset J'$ is a point of $\Hilb_{K\subseteq J}$. 
		Let $\pi\colon \Hilb_I \to \Hilb_K$ be the morphism given on closed points 
		by $I'\mapsto I'+L$. Let $V= \pi^{-1}(U)$ and pick $I'\in V$. 
		Let $K' = I' + L$ and $J'$ be such that
		$K'\subseteq J'$ is a point in $\Hilb_{K\subseteq J}$. 
		Let $\kappa'$ be the residue field at $[K']$. Since 
		$K'_{\uu+k_i\cdot \vv_i} \neq (S_{\kappa'})_{\uu+k_i\cdot \vv_i}$
		for all $i$ we conclude that $I'_\uu = K'_\uu \subsetneq J'_\uu$ and 
		$I'_{\uu+k_i\cdot \vv_i} = K'_{\uu+k_i\cdot \vv_i} = J'_{\uu + k_i \cdot \vv_i}$
		for all $i$. By part~\ref{it:2} there exists $f\in J'_\uu \setminus I'_{\uu}$
		and by part~\ref{it:3} it belongs to $(I')^{\sat}$.
		Therefore, $V$ consists of nonsaturated ideals. Thus, $I$ is nonsaturable.
		\end{proof}

		When $X$ is the product of projective spaces we obtain the following simplified statement.		

		\begin{corollary}\label{ref:def_criterion_product:cor}
		Let $X = \mathbb{P}^{n_1}\times \cdots \times \mathbb{P}^{n_s}$ and
        $S$ be its Cox ring (see Example~\ref{ex:productOfProjectives}).
		Let $I\subseteq \Sk$ be a $\Pic(X)$-homogeneous ideal and $A\subseteq \Pic(X)$ be a subset
		such that $A+\Eff(X) \subseteq A$. Let $L$ denote the ideal
        $\bigoplus_{\aa\in A} (\Sk)_\aa$
		and $K=I+L$. If there exist a $\Pic(X)$-graded
			ideal $J\subseteq \Sk$ a degree $\uu\in \mathbb{Z}^{s}_{\geq 0}$ and a positive integer $k$ such that
			\begin{enumerate}
				\item $K\subseteq J$,
				\item $\ObFib(K, J) = 0$,
				\item $K_\uu \subsetneq J_\uu$,
				\item $K_{\uu+k\cdot (1,1,\ldots, 1)} = J_{\uu + k \cdot (1,1,\ldots, 1)}
				 \subsetneq (\Sk)_{\uu + k\cdot (1,1,\ldots, 1)}$,
			\end{enumerate} 
			then $I$ is not saturable.
		\end{corollary}

    \subsection{Saturated locus in the multigraded Hilbert schemes}\label{ssec:smoothnessOfSaturated}

    In this subsection we study the saturated locus of the multigraded Hilbert scheme.

    As above, throughout this subsection, $X$ is a smooth projective toric variety over a Noetherian ring $\kk$
    with Cox ring $S$. The symbol $H$ denotes a Hilbert function admissible for the saturated ideal of a length $d$
    zero-dimensional subscheme of $X$. We use notation introduced
    in Subsection~\ref{ssec:MaclaganSmithHilbertScheme}.

    Let $P$ be the multigraded Hilbert polynomial of $H$. There is a natural map
        \[
            \Hilb^P(X)\to \Hilb^d(X),
        \]
        that sends an ideal $I\subseteq \Sk$ to the subscheme $V(I)\subseteq
        X_{\kappa}$, see~\cite[\S5.2]{cox_book} for the construction of $V(I)$.
		Furthermore, as discussed in Subsection~\ref{ref:toric_setup:subsec} there is a natural restriction morphism $\Hilb^H \to \Hilb^P(X)$ and therefore, we obtain a morphism
		\[
		\Hilb^H \to \Hilb^d(X).
		\]
		This allows us to identify the saturated locus in $\Hilb^H$ with a
        locally closed subscheme of the usual Hilbert scheme as follows.

        \begin{proposition}\label{ref:partialToGrothendieck:prop}
            Let $P$ be a multigraded polynomial equal to a constant $d$. Then,
            the map $\Hilb^P(X)\to \Hilb^d(X)$ is an isomorphism. In particular,
            for every Hilbert function $H$ with multigraded Hilbert polynomial $P$,
            the natural map $\Hilb^{H, \sat} \to \Hilb^d(X)$ is a locally closed
            immersion with image consisting of subschemes that
            have Hilbert function $H$.
        \end{proposition}
        \begin{proof}
            \def\II{\mathcal{I}}%
            It is enough to construct an inverse. Take a $\kk$-algebra $R$, an $R$-point of
            $\Hilb^d(X)$, the associated flat family $Z\subseteq
            X\times_{\kk} \Spec(R)$, and its ideal sheaf $\II_Z$. Let $I_Z\subseteq
            S_R$ be the homogeneous ideal of $Z$; it is obtained as $I_Z =
            \bigoplus_{\aa\in \Pic(X)} H^0(X_R, \II_Z(\aa))$.
            We 
            have an exact sequence
            \begin{align}\label{eq:H1}
                0&\to (I_Z)|_{\Nef(X)} \to S_R|_{\Nef(X)} \to \bigoplus_{\aa\in
                \Nef(X)} H^0(X_R,\OO_Z(\aa)) \to\\
                &\to\bigoplus_{\aa\in \Nef(X)} H^1(X_R,\II_Z(\aa))\to
                \bigoplus_{\aa\in \Nef(X)} H^1(X_R,\OO_Z(\aa))\nonumber.
            \end{align}
            Fix any $\aa\in \nn+\Nef(X)$, where $\nn$ was defined
            above~\eqref{eq:constantRank}.
            If $R$ is a field, then~\eqref{eq:constantRank} implies that
            $(S_R/I_Z)_{\aa}$ has rank $d$, the same as $H^0(X_R,\OO_Z(\aa))$.
            Moreover, in this case $H^1(X_R, \OO_Z(\aa))$ is zero,
            see~\cite[Theorem~9.2.3]{cox_book} for $R = \mathbb{C}$ and
            \cite[Theorem~3.6]{Altmann_Buczynski_Kastner_Winz} in general.
            This implies that $H^1(X_R,\II_Z(\aa))$ is zero in this case.
            For general $R$, by Cohomology and Base
            Change~\cite[Lemma~5.1.1]{Conrad__Grothendieck_duality} applied to
            the sheaf
            $\II_Z(\aa)$, we obtain that $H^j(X_R,\II_Z(\aa)) = 0$ for every $j\geq 1$.
            Given that, the sequence~\eqref{eq:H1} becomes short exact and $R\mapsto
            (I_Z)|_{\nn + \Nef(X)}$ yields the required inverse map $\Hilb^d(X)\to
            \Hilb^P(X)$ and thus the isomorphism. The final claims follow from
            Corollary~\ref{ref:fixedHilbertFunctionLocallyClosed:cor}.
        \end{proof}

    Recall that $\Hilb^d(X)$ has a distinguished component $\Hilb^{d, \sm}(X)$,
    called the \emph{smoothable component}. The general point $[\Gamma]$ of
    this component corresponds to a reduced $\Gamma$, that is, to a tuple of $d$ points of $X$.
    Even more, the set of points $[\Gamma]\in \Hilb^d(X)$ such
    that $\Gamma$ is smooth, is open. For the purposes of this discussion, we
    call it the \emph{locus of smooth subschemes} of $\Hilb^d(X)$.
    \begin{definition}\label{ref:Slip:def}
        The \emph{Scheme of Limits of Ideals of Points} $\Slip^H\subseteq
        \Hilb^H$ is the closure of the open subscheme: the intersection of
        $\Hilb^{H, \sat}$ and the preimage of the locus of smooth subschemes under the natural morphism
        $\Hilb^H \to \Hilb^d(X)$ constructed in Subsection~\ref{ref:toric_setup:subsec}.
        For $[I]\in \Hilb^H$, we also denote $\Slip^H$ by $\Slip_I$.
    \end{definition}
    By definition, a general point of $\Slip^H$ corresponds to a saturated
    ideal $I^{\sat}$ such that $V(I^{\sat})$ is a tuple of points with Hilbert function $H$.
    The map $\Slip^H\to \Hilb^{d, \sm}(X)$ can be described a
    bit more precisely when $H$ is the Hilbert function of a general $d$-tuple
    of points in $X$, that is, $H(\aa) = \min(d, \mathrm{rank}_\kk S_{\aa})$
for every $i$.
    For a zero-dimensional subscheme $\Gamma$ in a smooth quasi-projective
    $\kappa$-scheme
    $Y$, we say that
    $\Gamma$ is \emph{unobstructed} if $[\Gamma]\in
    \Hilb^d(Y)$ is a smooth point, where $d = \dim_{\kappa} H^0(\OO_{\Gamma})$ and
    $\Hilb^d(Y)$ is a Grothendieck Hilbert scheme. Unobstructedness is an
    intrinsic property of $\Gamma$, it does not depend on the embedding.

    \begin{corollary}\label{ref:exceptionalIsDivisor:cor}
        Let $H$ be the Hilbert function of a general $d$-tuple of points in
        $X$. Let $[I]\in \Slip^H$ be such that $V(I^{\sat})\subseteq
        X_{\kappa}$ is unobstructed and the fiber of the
        map $\Slip^H|_{\kappa}\to \Hilb^d(X|_{\kappa})$ over $V(I^{\sat})$ is not a point. Then there is a divisor $E\subseteq
        \Slip^H|_{\kappa}$ with $[I]\in E$ such that the image of $E$ in
        $\Hilb^d(X)|_{\kappa}$ has codimension at least two.
    \end{corollary}
    \begin{proof}
        The locally closed immersion from
        Proposition~\ref{ref:partialToGrothendieck:prop} defines a birational
        morphism $\Slip^H|_{\kappa} \to \Hilb^{d, \sm}(X|_{\kappa})$. Indeed, its restriction to
        the inverse image of the locus $\mathcal{V}$ of all smooth subschemes with Hilbert
        function $H$ is a surjective locally closed immersion onto the smooth
        scheme $\mathcal{V}$ and thus an isomorphism.
        The claim follows from applying~\cite[Theorem~2, p.120]{Shafarevich_1}
        to this map.
    \end{proof}

        \subsection{Fiber obstruction group for saturation}\label{ref:obfib_n_graded:sec}
        \newcommand{\cokern}[1]{C(#1^{\sat})}

        In this subsection we analyse the obstruction group further.
        We keep the toric setup as in the previous
        subsections, however in Subsection~\ref{ssec:fiberProjectiveSpace} we
        restrict to the projective space where more interesting results are
        obtained.

        Let $I$ be a homogeneous ideal in $\Sk$.
        For the pair $(K, J) := (I, I^{\sat})$, the fiber obstruction group
        from~\eqref{eq:obfib}
        can be ``underived'' as follows. Choose a transverse element $\ell$ as
        in Subsection~\ref{ssec:saturations}
        and let $\cokern{I}$ be defined by the exact sequence
        \begin{equation}\label{eq:cokernel}
            0\to \frac{\Sk}{I^{\sat}}\to \left( \frac{\Sk}{I^{\sat}} \right)_{\ell}\to
            \cokern{I}\to 0.
        \end{equation}
        Applying $\Hom(I^{\sat}/I, -)$ we obtain a long exact sequence.
        The multiplication by a large enough power of $\ell$ annihilates
        $I^{\sat}/I$ and is an isomorphism on $(\Sk/I^{\sat})_{\ell}$, hence
        $\Ext^{\bullet}_{\Sk}(I^{\sat}/I, (\Sk/I^{\sat})_{\ell}) = 0$ and the long exact
        sequence induces a homogeneous isomorphism
        \begin{equation}\label{eq:fiberObstructionConcrete}
            \Ext^1_{\Sk}\left( \frac{I^{\sat}}{I}, \frac{\Sk}{I^{\sat}}
            \right)  \simeq \Hom\left( \frac{I^{\sat}}{I},
            \cokern{I} \right).
        \end{equation}

        We can refine the above taking into account degrees. For $\midx, \ee\in
        \AA = \Pic(X)$ we say that $\midx\leq \ee$ if $\ee - \midx$ lies in
        the effective cone of $X$.
		Fix degrees $\midx \leq \ee$.
        Applying $\Hom(I^{\sat}_{\geq \mm}/I_{\geq \ee}, -)$ to $0\to
        (\Sk/I^{\sat})_{\geq \midx} \to ((\Sk/I^{\sat})_{\ell})_{\geq \midx} \to
        \cokern{I}_{\geq \midx}\to 0$ we arrive
        at the homogeneous isomorphism
        \begin{equation}\label{eq:gradedfiberObstructionConcrete}
            \Ext^1_{\Sk}\left( \frac{I^{\sat}_{\geq \midx}}{I_{\geq \ee}},
            \left(\frac{\Sk}{I^{\sat}}\right)_{\geq \midx} \right) \simeq \Hom\left( \frac{I^{\sat}_{\geq
             \midx}}{I_{\geq \ee}}, \cokern{I}_{\geq \midx} \right).
        \end{equation}
        By construction, the module $M := I^{\sat}_{\geq \midx}/I_{\geq \ee}$ is
        generated in degrees $\geq \midx$. The same is true for modules in the minimal resolution
        of $M$ since the ring $\Sk$ is positively
        graded. The Ext groups are computed from this resolution, so if we
        restrict to homogeneous maps, we can
        drop the subscripts in~\eqref{eq:gradedfiberObstructionConcrete} to obtain
        \begin{equation}\label{eq:fiberObstructionConcreteRestricted}
            \Ext^1_{\Sk}\left( \frac{I^{\sat}_{\geq \midx}}{I_{\geq \ee}},
            \frac{\Sk}{I^{\sat}_{\geq \midx}} \right)_{\zerodeg} \simeq \Hom\left( \frac{I^{\sat}_{\geq
            \midx}}{I_{\geq \ee}}, \cokern{I_{\geq \ee}} \right)_{\zerodeg}.
        \end{equation}

        \subsubsection{Case of the projective
        space}\label{ssec:fiberProjectiveSpace}
        \newcommand{\Homstar}{{}^*\!\Hom}%
		In this subsection we assume that $X$ is a projective space, so that
        $S$ becomes standard $\mathbb{N}$-graded polynomial ring. As always,
        take $I \subseteq \Sk$.
		In this setup the $\Sk/I^{\sat}$-module $\cokern{I}$ is pleasantly explicit. If $\dim
        \Sk/I^{\sat} = 1$ it admits connections to the canonical module
        $\omega_{\Sk/I^{\sat}}$~\cite[\S13.1]{24hours}, which
        are useful since the latter is sometimes even easier to describe. We discuss
        this connection now.

        Recall the graded dual operator: for a graded $\kappa$-vector space $M =
        \bigoplus_{e} M_{e}$ with every $M_{e}$ finite-dimensional we define $\Homstar(M, \kappa) :=
        \bigoplus_{e} \Hom(M_{-e}, \kappa) \subseteq \Hom(M, \kappa)$. The
        inclusion is an equality if $M$ has finite dimension but not in
        general. If $M$ is an $\Sk$-module, then $\Homstar(M, \kappa)$ is an
        $\Sk$-module as well by $(s\cdot \varphi)(m) := \varphi(sm)$ and the
        canonical map $M\to \Homstar(\Homstar(M, \kappa), \kappa)$ is an isomorphism
        of $\Sk$-modules.

        \begin{proposition}\label{ref:CIAndcanonicalModule:prop}
            Suppose that
            $\Sk/I^{\sat} = 1$. Then the module $\cokern{I}$ is isomorphic to
            $\Homstar(\omega_{\Sk/I^{\sat}}, \kappa)$ and to the top local cohomology
            of $\Sk/I^{\sat}$.
        \end{proposition}
        \begin{proof}
            Let $\ell_1$ be a transversal element for $I$.
            Since $\Sk/I^{\sat}$ is one-dimensional, we may pick $\dim \Sk - 1$
            elements $\ell_2, \ldots ,\ell_{\dim \Sk}$ of $(\Sk)_+$ that annihilate
            $\Sk/I^{\sat}$ and such that the radical of $(\ell_1, \ell_2, \ldots
            ,\ell_{\dim \Sk})$ is $(\Sk)_+$.
            Using properties of local cohomology~\cite[Proposition~7.3(b),
            Theorem~7.13]{24hours}, we conclude that $\cokern{I} \simeq
            H^1_{(\Sk)_+}(\Sk/I^{\sat})$. Since $\Sk/I^{\sat}$ is one-dimensional,
            by~\cite[Theorem~13.5]{24hours}
            we obtain $\Homstar(\cokern{I}, \kappa) \simeq \omega_{\Sk/I^{\sat}}$.
            Applying $\Homstar(-, \kappa)$ to this isomorphism we get the claim.
        \end{proof}
        The above Proposition~\ref{ref:CIAndcanonicalModule:prop} becomes
        easier in the Gorenstein case. Recall that $\Sk/I^{\sat}$ is
        \emph{Gorenstein} if for some (or every) transverse element $\ell$ the
        socle $\Sk/(I^{\sat}+(\ell))$ is
        one-dimensional~\cite[Proposition~21.5]{Eisenbud}.
        \begin{corollary}[Gorenstein case]\label{ref:CIForGorenstein:prop}
            Suppose that $\Sk/I^{\sat}$ is one-dimensional and Gorenstein. Then the module $\cokern{I}$ is isomorphic to
            $\Homstar(\Sk/I^{\sat}, \kappa)(-a)$ where $a$ is the largest index
            such that $H_{\Sk/I^{\sat}}(a) \neq H_{\Sk/I^{\sat}}(a+1)$.
        \end{corollary}
        \begin{proof}
            It follows from Grothendieck's local
            duality~\cite[Theorem~18.7]{24hours} and from
            Proposition~\ref{ref:CIAndcanonicalModule:prop} that $\cokern{I}
            \simeq \Homstar(\Sk/I^{\sat}, \kappa)(-a)$ for a certain $a$. It
            follows from this isomorphism that $a$ is the largest degree in
            which $\cokern{I}$ is nonzero. From~\eqref{eq:cokernel} it follows
            that this agrees with the $a$ in the statement.
        \end{proof}
        \begin{corollary}\label{ref:fiberObstructionForGorenstein:cor}
            Suppose that $I^{\sat}$ is one-dimensional and Gorenstein.
            Then
            \[
                \ObFib(I, I^{\sat}) \simeq
                \Hom_{\kappa}\left(\left(\frac{I^{\sat}}{I+(I^{\sat})^2}\right)_a,
                \kappa \right).
            \]
        \end{corollary}
        \begin{proof}
            Let $R = \Sk/I^{\sat}$. Using Corollary~\ref{ref:CIForGorenstein:prop}
            and~\eqref{eq:fiberObstructionConcrete} we obtain that
            \begin{align*}
                &\ObFib(I, I^{\sat})
                \simeq \Hom_{R}\left( \frac{I^{\sat}}{I+(I^{\sat})^2},
                \cokern{I} \right)_{0}
                \simeq \Hom_{R}\left( \frac{I^{\sat}}{I+(I^{\sat})^2},
                \Homstar(R, \kappa)(-a) \right)_{0}\\
                & \simeq \Hom_{R}\left( \frac{I^{\sat}}{I+(I^{\sat})^2},
                \Homstar(R, \kappa) \right)_{-a}\stackrel{(\star)}{\simeq}
                \Hom_{\kappa}\left(\frac{I^{\sat}}{I+(I^{\sat})^2}, \kappa
                \right)_{-a}  \simeq
                \Hom_{\kappa}\left(\left(\frac{I^{\sat}}{I+(I^{\sat})^2}\right)_a, \kappa
                \right),
            \end{align*}
            where $(\star)$ follows from the graded version of the tensor-Hom
            adjunction.
        \end{proof}

	\subsection{Smoothness of the saturated locus in the standard graded case}\label{ref:smoothness_of_saturated:sec}

    In this subsection we consider the toric setup but restrict to $X =
    \mathbb{P}^n$, so that $S$ is standard $\mathbb{N}$-graded.

    In this setup we provide
    classes of \emph{saturated} ideals which yield smooth points of the multigraded Hilbert schemes. 
    In addition to being interesting on its own, understanding these cases
	increases the applicability of Theorem~\ref{ref:saturationDeformWhenTangentSurjective:thm}.


    For a saturated homogeneous one-dimensional ideal $I^{\sat}\subseteq \Sk$ with a transverse element
    $\ell$ which is a linear form we fix a decomposition $\Sk = S'[\ell]$,
    where $S'$ is a polynomial subring. We define the one-parameter family for
    $(I^{\sat},\ell, S')$ as
    \[
        \begin{tikzcd}
            \Spec(\Sk/I^{\sat})\ar[rd, "\pi"']\ar[r, hook]  & \Spec(\Sk)  \simeq  \Spec(S')\times
            \mathbb{A}^1\ar[d, "\pr_2"]\\
            & \mathbb{A}^1 = \Spec(\kappa[\ell])
        \end{tikzcd}
    \]
    where $\pi$ is induced by the inclusion $\kappa[\ell]\to \Sk/I^{\sat}$.
    In the projective space $\Proj(\Sk)$, the quotient $\Proj(\Sk/I^{\sat})$ corresponds
    to a zero-dimensional subscheme disjoint from the hyperplane $(\ell = 0)$. It follows
    that for every nonzero $\lambda\in \kappa$ the scheme
    $\pi^{-1}(\lambda) = \Spec(\Sk/(I^{\sat}+(\ell-\lambda)))$ is isomorphic to $\Proj(\Sk/I^{\sat})$.
    \begin{lemma}\label{ref:family:lem}
        For every $(I^{\sat}, \ell, S')$ as above, the map $\pi$ is finite and flat.
    \end{lemma}
    \begin{proof}
        The $\kappa$-algebra $\Sk/(I^{\sat}+(\ell))$ is finitely generated and zero-dimensional, hence a
        finite-dimensional $\kappa$-vector space.
        Fix its basis $\mathcal{B}$. Since $\Sk/I^{\sat}$ is $\mathbb{N}$-graded,
        induction by degree proves that $\Sk/I^{\sat}$ is spanned by $\mathcal{B}$ as a
        $\kappa[\ell]$-module, so $\pi$ is finite. Suppose that there is a nonzero element of $\Sk/I^{\sat}$ annihilated
        by some nonzero polynomial in $\kappa[\ell]$. Taking leading forms, we
        find a nonzero element of $\Sk/I^{\sat}$ annihilated by a power of $\ell$, a
        contradiction with transversality of $\ell$. Hence, the $\kappa[\ell]$-module $\Sk/I^{\sat}$ is torsion-free, so
        $\pi$ is flat.
    \end{proof}
    \begin{proposition}\label{ref:extendingTangentVector:prop}
        Suppose that $\ell$ is a linear form transverse to a saturated
        one-dimensional ideal
        $I^{\sat}$ and that $\Spec(\Sk/(I^{\sat}+(\ell))\subseteq \Spec(\Sk)$ is unobstructed.
        Then the natural map
        \[
            \frac{\Hom_{\Sk}\left( I^{\sat}, \frac{\Sk}{I^{\sat}} \right)}{\ell\cdot \Hom_{\Sk}\left( I^{\sat}, \frac{\Sk}{I^{\sat}} \right)}\to \Hom_{\Sk}\left( I^{\sat},
            \frac{\Sk}{I^{\sat}+(\ell)}
            \right)
        \]
        is an isomorphism.
    \end{proposition}
    The theorem does not hold under the slightly weaker assumption that
    $\Proj(\Sk/I^{\sat}) \simeq \Spec(\Sk/(I^{\sat}+(\ell-1))\subseteq \Proj(\Sk)$ is
    unobstructed; one counterexample is a general tuple of $8$ points on
    $\mathbb{P}^3$.
    \begin{proof}
        By left-exactness of Hom, it is enough to prove that $\Hom_{\Sk}(I^{\sat},
        \Sk/I^{\sat})\to \Hom_{\Sk}(I^{\sat}, \Sk/(I^{\sat}+(\ell))$ is surjective.
        The family $\pi$ above provides a map $\widetilde{\pi}\colon \mathbb{A}^1\to
        \Hilb^d(\Spec(S'))$, where
        $d = \dim_{\kappa} \frac{\Sk}{I^{\sat}+(\ell)}$. By assumption, the point
        $\widetilde{\pi}(0)\in \Hilb^d(\Spec(S'))$ is smooth, hence a general
        point of the image is smooth as well. But all $\kappa$-points of
        $\widetilde{\pi}|_{\mathbb{A}^1\setminus \{0\}}$ correspond to different
        embeddings of the same abstract scheme $\Proj(\Sk/I^{\sat})$, hence we conclude that every point in the
        image of $\widetilde{\pi}$ is smooth. In particular, the cotangent sheaf
        of $\Hilb^d(\Spec(S'))$ pulled back to $\mathbb{A}^1$ is a
        \emph{free} sheaf of $\mathbb{A}^1$. As a result, we obtain
        surjectivity of the natural
        restriction map from
        the module of pulled back vector fields on $\Spec(\kappa[\ell])$ to such
        a module on $\Spec(\kappa[\ell]/(\ell))$.

        By Lemma~\ref{ref:vectorfields:lem} and a generalization
        of~\cite[Theorem~10.1]{Stromme_Hilbert}, we conclude that the
        $\kappa[\ell]$-module
        of (pulled back) vector fields for $\widetilde{\pi}$ is
        $\Hom_{\Sk}(I^{\sat}, \Sk/I^{\sat})$ and that this is a free $\kappa[\ell]$-module. Applying
        the same Lemma and reference
        to $\Spec(\kappa) = (\ell=0)\to \Hilb^d(\Spec(S'))$ we obtain
        $\Hom_{S'}(I^{\sat}/(I^{\sat}\cap (\ell), \Sk/(I^{\sat}+(\ell))$.
        As mentioned above,
        the restriction map of modules of vector fields
        \[
            \Hom_{\Sk}\left( I^{\sat}, \frac{\Sk}{I^{\sat}} \right)\to
            \Hom_{S'}\left(\frac{I^{\sat}}{I^{\sat}\cap(\ell)},\; \frac{\Sk}{I^{\sat}+(\ell)}
            \right)
        \]
        is surjective. Since $\ell$ is a nonzerodivisor in $\Sk/I^{\sat}$, we have
        $I^{\sat}\cap (\ell) = I^{\sat}\cdot (\ell)$, so the natural map
        \[
            \Hom_{\Sk}\left( I^{\sat}, \frac{\Sk}{I^{\sat}+(\ell)} \right)\to \Hom_{S'}\left(
            \frac{I^{\sat}}{I^{\sat}\cap (\ell)},\; \frac{\Sk}{I^{\sat}+(\ell)}
            \right)
        \]
        is bijective. It follows that the restriction $\Hom_{\Sk}(I^{\sat}, \Sk/I^{\sat})\to
        \Hom_{\Sk}\left(I^{\sat}, \Sk/(I^{\sat}+(\ell))\right)$ is surjective as well.
    \end{proof}

    \begin{corollary}\label{ref:zeroMap:cor}
        Under the assumptions of
        Proposition~\ref{ref:extendingTangentVector:prop}, the multiplication
        by $\ell$ on $\Ext^1_{\Sk}(I^{\sat}, \Sk/I^{\sat})$ is injective. Consequently, for every
        subideal $K \subseteq I^{\sat}$ such that $I^{\sat}\cdot \ell^{\gg 0} \subseteq K$ the natural map
        $\Ext^1_{\Sk}(I^{\sat}/K, \Sk/I^{\sat})\to \Ext^1_{\Sk}(I^{\sat}, \Sk/I^{\sat})$
        is zero.
    \end{corollary}

    \begin{proof}
        Applying $\Hom_{\Sk}(I^{\sat}, -)$ to the short exact sequence
        \[
            0\to \frac{\Sk}{I^{\sat}}(-1)\to \frac{\Sk}{I^{\sat}}\to \frac{\Sk}{I^{\sat}+(\ell)}\to 0
        \]
        and using
        Proposition~\ref{ref:extendingTangentVector:prop} we obtain that the
        multiplication by $\ell$ on $\Ext_{\Sk}^1(I^{\sat}, \Sk/I^{\sat})$ is injective. The
        multiplication by $\ell$ on $I^{\sat}/K$ is nilpotent, so it is
        nilpotent on
        $\Ext^1_{\Sk}(I^{\sat}/K, \Sk/I^{\sat})$ as well, whence the claim.
    \end{proof}

    \begin{theorem}[unobstructedness of Artinian reduction implies smoothness
        in multigraded setting]\label{ref:smoothPoints:thm}
        In this theorem we assume that $\kk = \kappa$ is a field (but we keep
        using $\kappa$ for notational consistency).
        Let $\Sk$ be a standard graded polynomial ring and $I^{\sat}\subseteq
        \Sk$ be a saturated one-dimensional
        homogeneous ideal with transverse element $\ell$ which is a linear
        form. If $\Spec(\Sk/(I^{\sat}+(\ell))\subseteq \Spec(\Sk)$ is unobstructed, then
        $[I^{\sat}]\in \Hilb_{I^{\sat}}$ is a smooth point.

    \end{theorem}
    \begin{proof}
        Let $d = \dim_{\kappa} \Sk/(I^{\sat}+(\ell))$ and fix $e\gg 0$.
        Consider the canonical map defined in~\S\ref{ssec:smoothnessOfSaturated}: 
        \[
            \begin{tikzcd}
                \Hilb_{I^{\sat}} \ar[r, "\simeq", "\pr_{I^{\sat}}^{-1}"'] &
                \Hilb_{I^{\sat}_{\geq e}\subseteq I^{\sat}} \ar[r, "\pr_{I^{\sat}_{\geq e}}"] &
                \Hilb_{I^{\sat}_{\geq e}} \simeq \Hilb^d(\mathbb{P}^{\dim \Sk-1}).
            \end{tikzcd}
        \]
        The natural map $\Hom_{\Sk}(I^{\sat}_{\geq e}, \Sk/I^{\sat}_{\geq e})_{0}\to
        \Hom_{\Sk}(I^{\sat}_{\geq e}, \Sk/I^{\sat})_{0}$ is surjective, hence
        Theorem~\ref{ref:obstructionAtFlag:thm} applies and shows that
        $[I^{\sat}_{\geq e}\subseteq I^{\sat}]$ has an obstruction group $\ObFlag$ given by
        \[
            \begin{tikzcd}
                \ObFlag\ar[r]\ar[d]\arrow[dr, phantom,
                "\usebox\pullback" , very near start, color=black] &
                \Ext^1_{\Sk}\left( I^{\sat},
                \frac{\Sk}{I^{\sat}} \right)_{0}\ar[d]\\
                \Ext^1_{\Sk}\left(I^{\sat}_{\geq e}, \frac{\Sk}{I^{\sat}_{\geq
                e}}\right)_{0}\ar[r] & \Ext^1_{\Sk}\left( I^{\sat}_{\geq e},
                \frac{\Sk}{I^{\sat}} \right)_{0}.
            \end{tikzcd}
        \]
        Using the one-parameter family $\pi$ and the fact that
        unobstructedness is open, we conclude that
        $\Spec(\Sk/(I^{\sat}+(\ell-1))\subseteq \Spec(\Sk)$ is unobstructed,
        hence also $\Proj(\Sk/I^{\sat}) \subseteq
        \Proj(\Sk)$ is unobstructed, so $[I^{\sat}_{\geq e}]\in \Hilb_{I^{\sat}_{\geq e}}$ is
        a smooth point. This means that for any obstruction
        at $[I^{\sat}_{\geq e}\subseteq I^{\sat}]$
        arising from any small extension, its image under the left vertical
        map is zero.
        By Corollary~\ref{ref:zeroMap:cor}, the right vertical map is
        injective.
        But then it follows that every such obstruction is zero,
        hence $[I^{\sat}]\in \Hilb_{I^{\sat}}$ is smooth.
    \end{proof}

    \begin{corollary}\label{ref:usualClassesSmooth:cor}
        Assume that $\kk = \kappa$ is a field.
        Suppose that $I^{\sat}\subseteq \Sk$ is a homogeneous
        saturated ideal and that one of the following holds
        \begin{enumerate}
            \item $\Sk$ is at most three-dimensional (so that $\Proj(\Sk) =
                \mathbb{P}^2$ or $\mathbb{P}^1$),
            \item $\Sk$ is four-dimensional (so that $\Proj(\Sk) \simeq
                \mathbb{P}^3$) and $\Sk/I^{\sat}$ is Gorenstein,
            \item $I^{\sat}$ is a complete intersection.
        \end{enumerate}
        Then $I^{\sat}\in \Hilb_{I^{\sat}}$ is unobstructed.
    \end{corollary}
    This corollary is
    classical and proven, for example, in the references given in the proof.
    \begin{proof}
        To prove smoothness, we may enlarge the base field $\kappa$, so we assume
        it is infinite. Then there exists a linear form $\ell$ transverse to $I^{\sat}$.
        The ideal $I^{\sat}/(I^{\sat}\cap (\ell))\subseteq \Sk/(\ell)$ is respectively
        \begin{enumerate}
            \item an ideal in $\Sk/(\ell) \simeq \kappa[\alpha_1, \alpha_2]$ or in
                $\kappa[\alpha_1]$,
            \item a Gorenstein ideal in $\kappa[\alpha_1, \alpha_2, \alpha_3]$,
            \item a complete intersection ideal.
        \end{enumerate}
        In these cases unobstructedness of $\Sk/(I^{\sat}+(\ell))$ follows from classical
        results: Hilbert-Burch Theorem~\cite[Theorem~1,
        Corollary~1]{Schaps__codimension_two}, Buchsbaum-Eisenbud
        theorem~\cite[Theorem~9.7]{HarDeform} and finally~\cite[Theorem~9.2]{HarDeform}, so
        the results follow from Theorem~\ref{ref:smoothPoints:thm}.
    \end{proof}

    \section{Examples}\label{sec:examples}

In this section we assume that $\Bbbk = \kappa$ is a field.

\subsection{The fiber obstruction group for cases with
$r=1$.}\label{ssec:ObFibGorenstein}

    Let $\Sk$ be standard $\mathbb{N}$-graded, with $\kappa$ infinite. Let $I\subseteq \Sk$ be
    homogeneous and let $a$ be the largest degree such that $\dim_{\kappa}
    (\Sk/I^{\sat})_a \neq \deg(\Sk/I^{\sat})$.
    As proven in Corollary~\ref{ref:fiberObstructionForGorenstein:cor} if
    $\Sk/I^{\sat}$ is Gorenstein, then $\ObFib := \ObFib(I,I^{\sat})$ is
    dual to the degree $a$ part of $I^{\sat}/(I+(I^{\sat})^2)$.

    \begin{example}[Line]\label{ref:necessaryConditionOnLine:exa}
        In the special case when $I^{\sat}$ is a degree $d$ subscheme of a
        line, we have $a=d-2$. In this case $H_{\Sk/I^{\sat}}(d-2) = d-1$,
        $H_{\Sk/I^{\sat}}(d-1) = d$ and $\ObFib = 0$ if and only if the
        ideals $I^{\sat}$ and $I+(I^{\sat})^2$ agree in degree $d-2$. Assuming additionally that $H_{\Sk/I}$ is
        nondecreasing, we get $\dim_{\kappa} (I^{\sat}/I)_{d-2} \leq 1$, hence $\ObFib \neq
        0$ only if $(I^{\sat})^2_{d-2} \subseteq I_{d-2}\subsetneq
        (I^{\sat})_{d-2}$.
    \end{example}

    \begin{example}[Line and some general points]\label{ex:lineAndPoints}
        Generalizing Example~\ref{ref:necessaryConditionOnLine:exa}, consider
        the case where $\Sk/I^{\sat}$ has Hilbert function $(1,e+2, e+3, e+4,
        \ldots , e + d-1, e+ d, e+d,
        \ldots )$ and there exists an $J\supseteq I$ with $\Sk/J^{\sat}$ having Hilbert
        function $(1,2,3,4, \ldots , d-1, d, d,  \ldots , \ldots )$.
        A choice of an element $\ell$ transversal both for $I$ and $J$ yields
        a surjective map $\cokern{I}\to \cokern{J}$. The spaces $\cokern{I}$,
        $\cokern{J}$ have the same
        Hilbert functions in positive degrees, so this map is bijective in positive
        degrees. The quotient $\Sk/J^{\sat}$ is Gorenstein, so we have that
        \[
            \ObFib(I)  \simeq \Hom_{\Sk}\left( \frac{I^{\sat}}{I},
            \cokern{I}_{\geq 1} \right)_{0}  \simeq \Hom_{\Sk}\left(
            \frac{I^{\sat}}{I}, \cokern{J}_{\geq 1}
            \right)_0  \simeq \Hom_{\kappa}\left(  \left(\frac{I^{\sat}}{I + I^{\sat}\cdot
            J^{\sat}}\right)_{d-2}, \kappa \right),
        \]
        which is nonzero if and only if $I^{\sat}$ and $I + I^{\sat}\cdot
        J^{\sat}$
        differ in degree $d-2$.
    \end{example}

    \subsection{Three points}\label{ssec:threePoints}
    Consider the example $I := I(\Gamma_0)'$ from the introduction. In this
    example $\ObFib$ is
    one-dimensional.

    We claim that all points from $\Hilb_I$ are saturable. The restriction of natural map
    $\Hilb_I \to \Hilb^3(\mathbb{P}^2)$ to the saturable locus is dominant and projective, hence
    onto $\Hilb^3(\mathbb{P}^2)$. If $[I']\in \Hilb_I$ is a nonsaturated ideal then $\Sk/(I')^{\sat}$ has Hilbert function $(1,2,3,3, \ldots)$.
    Therefore, $[I']$ is the unique point of $\Hilb_{I}$ whose saturation is
    $(I')^{\sat}$.
    It follows that $I'$ is saturable.
    This claim follows also from \cite[Thm.~1.3]{Mandziuk}.

    \subsection{Four points}\label{ssec:four}

    Suppose that $I$ is a nonsaturated homogeneous ideal in a standard graded polynomial
    ring $\Sk = \kappa[\alpha_0, \ldots ,\alpha_3]$ such that the Hilbert function
    of $\Sk/I$ is $(1,4,4,4, \ldots )$, where \emph{$\kappa$ has characteristic
    zero}. The Hilbert function of $\Sk/I^{\sat}$
    is $(1,3,4,4,\ldots)$ or $H=(1,2,3,4,4, \ldots)$.
    In the first case, $[I]$ is the unique point in $\Hilb_I$
    whose saturation is equal to $I^{\sat}$. Therefore, we get that $I$ is
    saturable arguing as in~\S\ref{ssec:threePoints}. We assume that $\Sk/I^{\sat}$ has Hilbert function~$H$. 
    By Theorem~\ref{ref:saturationDeformsWhenObFibVanishes:thm} and
    Example~\ref{ref:necessaryConditionOnLine:exa} a necessary condition for $I$ to be
    saturable is that $(I^{\sat})^2_2 \subseteq I$. We claim that it is also a sufficient
    condition.

    Let $Z$ in $\Hilb^4(\mathbb{P}^3)$ consist of subschemes with Hilbert
    function $H$. This is a closed, irreducible and $8$-dimensional locus
    parameterized by a line in $\mathbb{P}^3$ and a quartic on it.
    Let $Y$ in $\Hilb_I$ be the preimage of $Z$. The fiber of $Y\to Z$ over any $[R]$
    is $\Gr(6, I(R)_2) \cong \mathbb{P}^6$, so $Y$ is irreducible and
    $14$-dimensional.
    Let $Y'\subseteq Y$ be the locus consisting of $[I']$
    such that $((I')^{\sat})^2_2\subseteq I'$. The fiber of $Y'\to Z$
    over any $[R]$ is isomorphic with
	$\Gr(3, I(R)_2/I(R)^2_2) \cong \mathbb{P}^3$, so $Y'$
    is irreducible and $11$-dimensional.

    Let $X = \Slip^{(1,4,4, \ldots )}$ be the locus of saturable ideals in $\Hilb_I$. It is irreducible
    and $12$-dimensional.
    By
    Example~\ref{ref:necessaryConditionOnLine:exa} and
    Theorem~\ref{ref:saturationDeformsWhenObFibVanishes:thm}, all saturable
    ideals from $Y$ lie in $Y'$, so $X \cap Y \subseteq Y'$. We now prove that
    this containment is an equality.
	Take any line and four points $\Gamma$ on it. Converging four general points
        to $\Gamma$ yields an ideal $[I']\in X\cap Y$ with $V(I'^{\sat}) =
        \Gamma$; in this way we obtain more than one $I'$ for given $\Gamma$. The
        scheme $V(I'^{\sat})$ is a complete intersection, hence
        unobstructed. Applying Corollary~\ref{ref:exceptionalIsDivisor:cor} we
        obtain a divisor $[I']\in E \subseteq X$, so $\dim E = 11$. Since
    $E$ gets contracted in $\Hilb^4(\mathbb{P}^3)$, the saturation of a point of $E$ cannot
    have Hilbert function $(1,3,4, \ldots )$, thus is has
    Hilbert function $H$, so $E \subseteq X \cap
    Y\subseteq Y'$. But $Y'$ is irreducible and $\dim Y' = 11 = \dim
    E$, so $Y' = E \subseteq X$.

    We have thus proven the following.
    \begin{proposition}
        Let $I\subseteq
        \kappa[\alpha_0,  \ldots , \alpha_3]$ be a nonsaturated ideal such that $H_{S/I} =
        (1,4,4,4, \ldots )$. Then $I$ is a limit of saturated ideals if and
        only if $\ObFib(I, I^{\sat})\neq 0$ if and only if $(I^{\sat})^2_2
        \subseteq I$.
    \end{proposition}

    \subsection{Five points}\label{ssec:five}
    Suppose that $I$ is a (not necessarily saturated) homogeneous ideal in a standard graded polynomial
    ring $\Sk = \kappa[\alpha_0, \ldots ,\alpha_4]$ such that the Hilbert function
    of $\Sk/I$ is $(1,5,5,5, \ldots )$. \emph{We assume that $\kappa$ has
    characteristic zero.} There are a few possible Hilbert
    functions of $\Sk/I^{\sat}$. We divide them according to the number of
    quadrics.

    \subsubsection{Case $H_{\Sk/I^{\sat}}(2) = 5$.} In this case $[I]$ is the unique point
    of the fiber of the natural map $\Hilb_I\to \Hilb^5(\mathbb{P}^4)$ defined
    in~\S\ref{ssec:setup}, hence
    is saturable.

    \subsubsection{Case $H_{\Sk/I^{\sat}}(2) = 4$.}\label{ssec:fivePts4Quadrics} Using Macaulay's Growth
    Theorem as in~\cite[Lemma~2.9]{cjn13} we conclude that $H_{\Sk/I^{\sat}} =
    (1, 3, 4, 5, 5, \ldots )$.
    Geometrically, the scheme $V(I^{\sat}_{\leq 3})\subseteq \mathbb{P}^4$ is a line
    and a (possibly embedded) point. Let $J\supseteq I^{\sat}_{\leq 3}$ define this
    line. The locus $\mathcal{L}\subseteq \Hilb^5(\mathbb{P}^4)$ of
    possible $V(I^{\sat})$ is parameterized by the choice of the
    point in $\mathbb{P}^4$, the line and four points on it, hence it is
    irreducible and its dimension is $4+6+4=14$. Every $V(I^{\sat})$ in this
    locus is contained in $\mathbb{P}^2$, hence unobstructed.

    The ideal $I$ contains no elements of degree less than two, so the $i$-th
    element $F_i$ of its minimal free resolution $ \ldots \to F_1\to F_0\to I\to 0$ has generators of degree at least
    $2+i$. In particular, every generator of $F_1$ has degree at least three.
    Since $H_{I^{\sat}/I} = (0,2,1,0, \ldots )$ this shows that
    $\Ext^1_{\Sk}(I, I^{\sat}/I)_0 = 0$ and so $\psi_0$ is surjective.

    From~Example~\ref{ex:lineAndPoints}, we see that $\dim_{\kappa} \ObFib(I, I^{\sat})$
    is zero when $(I^{\sat})_1\cdot J_1\not\subset I_2$.
    By Theorem~\ref{ref:saturationDeformsWhenObFibVanishes:thm}, the ideal $I$ is saturable only if
    $(I^{\sat})_1\cdot J_1\subseteq I_2$.
    Let $\mathcal{U}'$ be the preimage of $\mathcal{L}$ in $\Hilb_{I\subseteq I^{\sat}}$.
    For every $[I\subseteq I^{\sat}]\in \mathcal{U}'$ the space
    $I^{\sat}_1\cdot J_1$ is $5$-dimensional and it varies continuously with
    $I\subseteq I^{\sat}$, hence we obtain a closed subvariety $\mathcal{V}
    \subseteq \mathcal{U}'$ whose fiber over $[I^{\sat}]$ is
    given by $[I]$ with $I^{\sat}_1\cdot J_1\subseteq I_2$; this fiber is
    isomorphic to $\Gr(5, I^{\sat}_2/I^{\sat}_1\cdot J_1)$, hence is
    $5$-dimensional.
    By~\cite[11.4.C, 25.2.E]{Vakil_foag} the variety $\mathcal{V}$ is smooth,
    irreducible and $19$-dimensional, hence also
        $\pr_I(\mathcal{V})$ is irreducible and $19$-dimensional.

    The locus $\mathcal{V}$ contains a point $[I'\subseteq (I')^{\sat}]$ with $I'$ saturable: when
    a general $5$-tuple of points $\Gamma_t$ deforms to $\Gamma$ a $4$-tuple on a line and fifth
    point is ``static'' outside it, the limit $I'$ of $I(\Gamma_t)$ is one such
    ideal; for given $\Gamma$ one obtains more than $I'$
        (compare~\S\ref{ssec:four}).
    Let $X = \Slip^{(1,5,5, \ldots)}$. By Corollary~\ref{ref:exceptionalIsDivisor:cor} there is a
    $19$-dimensional family $[I']\in E \subseteq X$,
    which gets contracted in $\Hilb^5(\mathbb{P}^4)$. The saturation of a
    general element $[I'']\in E$ satisfies $H_{\Sk/I''}(2) \geq 4$ by
    semicontinuity and $H_{\Sk/I''}(2) \neq 5$ since $E$ gets contracted.
    Therefore, $E \subseteq \pr_I(\mathcal{U'})\cap X \subseteq
    \pr_I(\mathcal{V})$. Since $\pr_I(\mathcal{V})$ is irreducible and
    $19$-dimensional, we conclude that $E = \pr_I(\mathcal{V})$.

    \subsubsection{Case $H_{\Sk/I^{\sat}}(2) = 3$.} Using Macaulay's Growth
    Theorem as in~\cite[Lemma~2.9]{cjn13} we conclude that $H_{\Sk/I^{\sat}} =
    (1, 2, 3, 4, 5, 5, \ldots )$. When $(I^{\sat})^2_3\not\subseteq I$, we
    obtain by Example~\ref{ref:necessaryConditionOnLine:exa} that $\ObFib = 0$, hence $I$ is entirely nonsaturable.

    The space of possible $I^{\sat}$ has dimension $\dim\Gr(3, (\Sk)_1) + \dim
    \mathbb{P}(\kappa[\alpha_3, \alpha_4]_5) = 11$.
    Constructing $I$ from given $I^{\sat}$ is a bit more tricky in the current
    subcase. Each such ideal requires a choice of a $10$-dimensional space of
    quadrics inside a $12$-dimensional space $(I^{\sat})_2$, but a general
    choice $\mathcal{Q}$ of such quadrics will yield $H_{\Sk/(\mathcal{Q})}(3)
    <5$, so does not give rise to $I$.
    Instead of dealing with ideals directly, we employ Macaulay's inverse
    systems~\S\ref{ssec:macaulaysInverseSystems} in their partial
    differentiation flavour.

    Up to coordinate change, we have $I^{\sat}_1 = (\alpha_0, \alpha_1,
    \alpha_2)$, so in terms of inverse systems, for every $e\leq 4$ the space
    $(I^{\sat})^{\perp}_e$ is $\kappa[x_3, x_4]_e$.  To construct $I$, we need
    to add a one-dimensional space of cubics $\spann{c}$ and a 2-dimensional
    space of quadrics $\spann{q_1, q_2}$, taking into account that the
    derivatives of the cubic lie in the $5$-dimensional space $\spann{x_3^2,
    x_3x_4, x_4^2, q_1, q_2}$.

    We construct two families. First, the \emph{cubic only} family is obtained
    by taking a cubic $c$ whose space of partial derivatives $(\Sk)_1\circ c$
    is two-dimensional modulo $\kappa[x_3, x_4]_2$. The condition $(I^{\sat})^2_3
    \subseteq I$ is equivalent to $(\alpha_0, \alpha_1, \alpha_2)^2\circ c = 0$,
    which in turn implies that
    \[
        c = x_0q_0 + x_1q_1 + x_2q_2 + c_0,
    \]
    where $q_0, q_1, q_2\in \kappa[x_3, x_4]_2$ and $c_0\in \kappa[x_3, x_4]_3$.
    The choice of $c_0$ does not affect the ideal $I$, hence we have a $(3\cdot
    3 - 1)$-dimensional choice of $c$. Altogether, we obtain a
    $19$-dimensional irreducible family.
    We now prove that it lies in $\Slip_I$. It is enough to prove this for a
    general point of this family. Passing to the algebraic closure if
    necessary, we assume that
    $V(I^{\sat}) = \{[\mu_1], \ldots ,[\mu_5]\} \subseteq \mathbb{P}\spann{x_3,
    x_4}$, where $[\mu_i]\neq [x_3]$ for $i=1,2 \ldots ,5$.
    Since $H_{\Sk/I^{\sat}}(2) = H_{\kappa[x_3,x_4]}(2)$, we have $\spann{\mu_1^2, \ldots
    ,\mu_5^2} = \kappa[x_3,
    x_4]_2$. Rescaling $\mu_i$ if necessary, we assume $\sum_{j=1}^5 \mu_j^3 =
    0$. Fix $c_1, c_2, c_3\in \kappa^5$ such that $\sum_{j=1}^5 (c_{i})_j
    \mu_j^2 = q_i$ for $i=0,1,2$.
	For $j=1,\ldots, 5$ let $p_j = \mu_j + t\sum_{i=0}^2 (c_i)_j x_i$ and consider $
	\Gamma_t = \left\{ p_j \ |\ j=1, \ldots ,5 \right\}\subseteq\mathbb{P}^4.$
	By Example~\ref{ex:points}, for every $e > 0$ and nonzero $t$ we have
	$I(\Gamma_t)^{\perp}_e = \spann{p_j^e\ |\ j=1,2, \ldots ,5}$.
    We now use Proposition~\ref{ref:HilbertFunctionFromLimits:prop} to show that for every $e>0$ and general 
    $\lambda$, the vector space $I(\Gamma_\lambda)^{\perp}_e$ is five-dimensional and that the limit of those
    spaces is $I^\perp_e$.
	In the notation of that Proposition, for every $e>0$ and $i\in \{1,\ldots, 5\}$ we take $F_i = p_i^e$.
	If $e>3$, then $\mu_1^e, \ldots, \mu_5^e$ are linearly independent so we
    may take $G_i=F_i$ for all $i$ and conclude using that Proposition.
	Let
	\[
	H = \frac{1}{3t}\sum_{j=1}^5\left( \mu_j + t\sum_{i=0}^2 (c_i)_j x_i\right)^3.
	\]
	Observe that for every homogeneous partial
	derivative operator $D$ of order $k\leq 3$  we have
	\begin{equation}\label{eq:5points_in_the_span}
        DH \in \frac{1}{t}\spann{(p_1)^{3-k}, \ldots, (p_5)^{3-k}}\kappa[t] 
	\end{equation}
	and
	\begin{equation}\label{eq:derivatives_agree}
		\lim_{t\to 0 } (DH) = D(x_0 q_0 + x_1 q_1 + x_2 q_2) \equiv
		Dc \mod \kappa[x_3, x_4]_{3-k},
	\end{equation}
    if $D(x_0 q_0 + x_1 q_1 + x_2 q_2)\neq 0$.
	Assume that $e\in \{1,2,3\}$. Let $G_1, \ldots, G_{4-e}$ be $D_1 H, \ldots, D_{4-e} H$, where
	$D_i$ are homogeneous partial derivative operators in $x_3$ and $x_4$ of order $3-e$ such that 
	$D_1 (x_0 q_0 + x_1 q_1 + x_2 q_2), \ldots, D_{4-e} (x_0 q_0 + x_1 q_1 + x_2 q_2)$ are linearly independent.
	Let $G_{5-e}, \ldots, G_5$ be some of $F_1, \ldots, F_5$ such that the 
	corresponding $\mu_j^e$ are linearly independent.
	The above claims about $I(\Gamma_t)^\perp_e$ follow from Proposition~\ref{ref:HilbertFunctionFromLimits:prop},
	\eqref{eq:5points_in_the_span} and \eqref{eq:derivatives_agree}.
	We proved that for general $\lambda$ the Hilbert function of
    $\Sk/I(\Gamma_{\lambda})$ is $(1,5,5, \ldots )$ and the limit of
    $[I(\Gamma_{t})]\in\Hilb^{(1,5,5 \ldots )}$ with $t\to 0$ is the ideal
    $I$, so $[I]\in \Slip^{(1,5,5 \ldots )}$.

    Next, the \emph{cubic and quadric} family is obtained by taking the cubic
    $c$ to be a cube of a linear form and all possible limits of $c$.
    The choice of $c$ is $4$-dimensional. The only nontrivial possible limit
    is $\ell_1\cdot (\ell_2)^2$, where $\ell_1$ is a linear form and $\ell_2\in
    \spann{x_3, x_4}$. We call such cubics \emph{impure}.
    As for the choice of quadrics, we fix the quadric $q_1$ as any basis element of the
    one-dimensional space $(\Sk)_1\circ c$. The second quadric can be chosen
    arbitrarily modulo $\spann{x_3^2, x_3x_4, x_4^2, q_1}$, hence yields a
    $10$-dimensional choice.
    It follows that the whole family has dimension $11+10+4 = 25$.
    In particular, its general element does not lie in $\Slip_I$.
    Moreover, from the condition $(I^{\sat})^2_3\subseteq I$ we see that the
    only elements that are possibly in $\Slip_I$ have impure
    $c$. Hence we assume that $c=x_0x_3^2$.

    We consider the subcase when $q_2(x_0, x_1, x_2, 0, 0) = \nu x_0^2$
    for some $\nu\in \kappa$, hence
    $q_2 = \nu x_0^2 + x_3\ell_1 + x_4\ell_2$ for some $\ell_1, \ell_2\in \spann{x_0,
    \ldots ,x_4}$.
    In this case we claim that the ideal $I$ is in $\Slip_I$ for every choice
    of $\nu$ and $q_2$ and the quintic in $I^{\sat}$. Since $\Slip_I$ is
    closed, it is enough to
    prove this for a general choice and over an algebraically closed field
    $\kappa$, thus we assume that the quintic is completely decomposable and
    $V(I^{\sat}) = \{[\mu_1], \ldots ,[\mu_5]\} \subseteq \mathbb{P}\spann{x_3,
    x_4}$, where $[\mu_i]\neq [x_3]$ for $i=1,2 \ldots ,5$.
    The five cubes $\mu_1^3$, \ldots ,$\mu_5^3$ are linearly dependent as they
    lie in a four-dimensional space $\kappa[x_3, x_4]_3$. Rescaling the forms $\mu_i$, we assume
    $\sum_{i=1}^5 \mu_i^3 = 0$. Moreover, since $H_{\Sk/I^{\sat}}(2) = 3$, the squares $\mu_1^2$, \ldots ,$\mu_5^2$
    span $\kappa[x_3, x_4]_2$.
    Take $c_1, c_2, c_3, c_4\in \kappa^5$ such that
    \[
        \sum_{i=1}^5 c_{1i}\mu_i^2 = x_3^2,\qquad \sum_{i=1}^5 \left(
        \frac{\partial\mu_i}{\partial x_4}
        \right)c_{1i}^2 \neq 0
	\]
	and
	\[
		\sum_{i=1}^5 \mu_i \left(
        \frac{\partial \mu_i}{\partial x_4} \right)(c_{2i}x_0 + c_{3i}x_1 + c_{4i}x_2) \equiv
        \frac{\sum_{i=1}^5 \left(\frac{\partial \mu_i}{\partial x_4}\right) c_{1i}^2}{2\nu}(\ell_1x_3+\ell_2x_4) \mod 			\kappa[x_3, x_4]_{2}.
    \]
    To prove the existence of $c_1$, it is enough to prove it on a single
    example, for example with $[\mu_{\bullet}] = ([x_4 - x_3], [x_4+x_3],
    [x_4-2x_3], [x_4 + 2x_3], [x_4])$ and $c_1 = \frac{1}{8}(0, 0, 1, 1, -2)$.
	For $i=1,\ldots, 5$ let $p_i= \mu_i + tc_{1i}x_0 + t^2(c_{2i}x_0 + c_{3i}x_1 + c_{4i}x_2)$ 
	and consider the family $\Gamma_t =  \{p_i\ |\ i=1,2, \ldots ,5\}\subseteq \mathbb{P}^4$.
   	We show that $\Sk/I(\Gamma_\lambda)$ has Hilbert function $(1,5,5,\ldots)$ for general $\lambda$ and 
   	the limit of ideals $I(\Gamma_t)$ taken degree-by-degree (hence, in the
    multigraded Hilbert scheme) is $I$. As above, we use 
    Proposition~\ref{ref:HilbertFunctionFromLimits:prop} and for each degree 
    $e>0$ we take $F_i = p_i^e$ for each $i=1,\ldots, 5$. 
	For $e>3$ we may take $G_i=F_i$ for every $i$ and conclude using that Proposition. Let
	\[
	H = \frac{1}{3t} \sum_{i=1}^5 \left(\mu_i + tc_{1i}x_0 + t^2(c_{2i}x_0 + c_{3i}x_1 + c_{4i}x_2)\right)^3.
	\]
	For $e\in \{1,2,3\}$ we take $G_1,\ldots, G_{e+1}$ to be any $e+1$ of $F_1,\ldots, F_5$ such that the 
	corresponding $\mu_i^e$ are linearly independent.
	For $e=3$ we take $G_5 = H$ and conclude using Proposition~\ref{ref:HilbertFunctionFromLimits:prop} since 
	\[
	\lim_{t\to 0} H = \sum_{i=1}^5 \mu_i^2c_{1i}x_0 = x_0x_3^2.
	\]
	Before proceeding with cases $e=1,2$ observe that
	\[
	\lim_{t \to 0} \frac{1}{2t} \frac{\partial H}{\partial x_4} \equiv
	\left(\frac{1}{2}\sum_{i=1}^5 \frac{\partial \mu_i}{\partial x_4} c_{1i}^2 \right)
	\cdot \left(x_0^2 + \frac{\ell_1x_3}{\nu} + \frac{\ell_2 x_4}{\nu}\right) \mod \kappa[x_3,x_4]_2.
	\]	
	By assumption, the first factor is nonzero. Replacing $H$ by $H$ divided by that 
	scalar and multiplied by $\nu$ we get
	\[
	\lim_{t \to 0} \frac{1}{2t} \frac{\partial H}{\partial x_4} \equiv q_2 \mod \kappa[x_3,x_4]_2.
	\]
	Therefore, for $e=2$ we may take $G_4 = \frac{\partial H}{\partial x_3}$ and
	 $G_5 = \frac{1}{2t} \frac{\partial H}{\partial x_4}$.
	Finally, for $e=1$ we take 
	$G_3 = \frac{1}{2t} \frac{\partial^2 H}{\partial x_4^2}$, $G_4 =
	\frac{1}{2t} \frac{\partial^2 H}{\partial x_4 \partial x_3}$ and
	 $G_5 = \frac{1}{2t} \frac{\partial^2 H}{\partial x_4 \partial x_0}$.
    This finishes the proof of our claims about $I(\Gamma_t)$ and implies
     that $[I] \in \Slip^{(1,5,5,\ldots)}$.

	Finally, we consider the subcase when $q_2(x_0, x_1, x_2, 0, 0)$ is not a
	multiple of $x_0^2$. Consider the natural map 
	$\pi\colon \Hilb_I \to \Hilb_{I+(\Sk)_{\geq 4}}$.
	The images of $\Slip_I$ and the cubic and quadric family are both $20$-dimensional.
	The tangent space at the image of any ideal corresponding to 
	$(q_2,c) = (x_0x_1, x_0x_3^2)$ is $20$-dimensional, so no such ideal is in $\Slip_I$.
	If the quadric $q_2(x_0, x_1, x_2, 0, 0)$ has rank at least
	two, then $q_2$ degenerates to $x_0x_1$ so the corresponding ideal is not in $\Slip_I$.
	We are left with the rank one case, so $q_2(x_0, x_1, x_2, 0, 0) = x_1^2$.
	It is enough to show that no ideal with $(q_2, c) = (x_1^2, x_0x_3^2)$ is in $\Slip_I$.
	Let $K$ be the image under $\pi$ of any such ideal. A computation using the \emph{VersalDeformations}
    package~\cite{Ilten_Versal_Deformations} verifies that $[K]$ lies on the intersection of two irreducible components: 
    $20$-dimensional and $19$-dimensional. It follows that $[K]\notin \pi(\Slip_I)$.	

    Summing up the five points case, we have the following:
    \begin{proposition}
        Let $I\subseteq
        \kappa[\alpha_0,  \ldots , \alpha_4]$ be an ideal such that $H_{\Sk/I} =
        (1,5,5,5, \ldots )$.
        \begin{enumerate}
            \item Suppose $H_{\Sk/I^{\sat}}(2) = 5$. Then $I$ is a limit of
                saturated ideals.
            \item Suppose $H_{\Sk/I^{\sat}}(2) = 4$.
                Then $I$ is a limit of saturated ideals if and
                only if $\ObFib(I, I^{\sat})\neq 0$ if and only if
                $(I^{\sat})_1\cdot J_1
                \subseteq I$, where $J$ was defined
                in~\ref{ssec:fivePts4Quadrics}.
            \item Suppose $H_{\Sk/I^{\sat}}(2) = 3$. Assume that $\kappa$ has characteristic zero.
                Then $I$ is a limit of saturated ideals if and only if both
                $(I^{\sat})^2_3 \subseteq I$ and $(I:(\Sk)_+^2)_1 \cdot
                I_1^{\sat} \subseteq I$.
        \end{enumerate}
    \end{proposition}

\subsection{A partially saturable ideal with nonzero fiber obstruction
group}\label{ssec:nonvanishingGroup}

    Let $\Sk = \kappa[\alpha_0, \alpha_1, \alpha_2]$ with standard grading and
    consider the degree $9$ ideal
    \[
        I = (\alpha_0^2\alpha_2,\, \alpha_0\alpha_1^3,\,
        \alpha_0^2\alpha_1^2,\, \alpha_0^3\alpha_1,\, \alpha_0^5,
        \alpha_1^6).
    \]
    Its Hilbert function is $(1,3,6,9, \ldots )$ and it has $I^{\sat} =
    (\alpha_{0}^{2},\,\alpha_{0}\alpha_{1}^{3},\,\alpha_{1}^{6})$, whose Hilbert function is
    $(1,3,5,7,8,9, \ldots )$.
    In this example we check that the map $\Hom_{\Sk}(I^{\sat}, \Sk/I^{\sat})_0\to
    \Hom_{\Sk}(I, \Sk/I^{\sat})_0$ has one-dimensional cokernel and the same holds for $\TgFlag\to
    \Hom_{\Sk}(I, S/I)_0$. Choosing a complementary tangent vector we discover the
    deformation
    \[
        I_t = (t\alpha_0\alpha_1^2 + \alpha_0^2\alpha_2,\, \alpha_0\alpha_1^3,\,
        \alpha_0^2\alpha_1^2,\, \alpha_0^3\alpha_1,\, \alpha_0^5, \alpha_1^6).
    \]
    The saturation of the ideal $I_1$ is $(\alpha_0\alpha_1^2 +
    \alpha_0^2\alpha_2, \alpha_0^2\alpha_1, \alpha_0^3, \alpha_1^6)$ which has
    Hilbert series $(1,3,6,7,8,9, \ldots )$. This shows that $I = I_0$ is not
    entirely nonsaturable. In fact, with the help of \emph{VersalDeformations}
    package~\cite{Ilten_Versal_Deformations} we check that $[I_1]\in
    \Hilb_{I_1}$ is smooth and has $\ObFib = 0$, so it is entirely
    nonsaturable. We check directly that the tangent space dimensions at $I_0$
    and $I_1$ are equal, so it also follows that $I$ is a smooth point on the
    component of $I_1$, so $I$ is nonsaturable as well.

\subsection{An entirely nonsaturable ideal with nonzero fiber obstruction
    group using \BBname{} decomposition}

    Let $\Sk = \kappa[\alpha_0, \alpha_1, \alpha_2]$ with standard grading and
    consider the ideal
    \[
        I = (\alpha_0^3, \alpha_0\alpha_1^2, \alpha_0^2\alpha_2,
        \alpha_0\alpha_1\alpha_2, \alpha_0\alpha_2^4, \alpha_2^6).
    \]
    In this case $\ObFib$ is one-dimensional, so
    Theorem~\ref{ref:saturationDeformsWhenObFibVanishes:thm} does not apply. We will
    nevertheless prove that $I$ is entirely nonsaturable.
    Refine the standard grading on $\Sk$ to an $\mathbb{N}^2$-grading, where
    $\deg(\alpha_0) = \deg(\alpha_2) = (1, 0)$ and $\deg(\alpha_1) = (1,1)$.
    With the help of \emph{Macaulay2} we check that for the pair $(I,
    I^{\sat})$ we have $\ObFib_{0, *} =
    \ObFib_{0, -1}$, so in particular $\ObFib_{0, \geq 0} = 0$.
    The second coordinate of the grading gives a $\kappa^*$-action on $\Hilb_I$
    and $\Hilb_{I\subseteq I^{\sat}}$.
    We now use the theory of \BBname{} decompositions. Introducing this theory
    would take too much place, so we only sketch the main results.
    Repeating the proofs of
    Theorem~\ref{ref:obstructionAtFlag:thm} and
    Theorem~\ref{ref:deformsWhenObFibVanishes:thm} for degrees $(0, \geq 0)$,
    we conclude that the map $\Hilb_{I\subseteq I^{\sat}}^+\to \Hilb_I^{+}$ on
    \BBname{} decompositions is smooth, hence also the map
    \[
        \id_{\mathbb{A}^2} \times \pr_{I}\colon \mathbb{A}^2 \times \Hilb_{I\subseteq I^{\sat}}^+\to
        \mathbb{A}^2 \times \Hilb_I^{+}
    \]
    is smooth.
    Consider the embedding
    \[
        \mathbb{A}^2\ni (v_1, v_2)\to \begin{pmatrix}
            1 & v_1 & 0\\
            0 & 1 & 0\\
            0 & v_2 & 1
        \end{pmatrix}\in \GL_3
    \]
    The scheme $\mathbb{A}^2 \times \Hilb_I^{+}$ maps to $\Hilb_I$
    by the forgetful map $\Hilb_I^+\to \Hilb_I$ followed by the
    $\mathbb{A}^2$-action via $\mathbb{A}^2\into \GL_3$.
    Under this map, the tangent space to $\mathbb{A}^2$ maps to a
    $2$-dimensional subspace of $\Hom_{\Sk}(I, \Sk/I)_{0,-1}$.
    A final \emph{Macaulay2} computation
    shows that this subspace is the degree $(0, <0)$ part of the
    tangent space, hence the above tangent map is surjective in degrees
    $(0, *)$. Arguing as in~\cite[Theorem~4.5]{Jelisiejew__Elementary},
    we conclude that $\mathbb{A}^2 \times \Hilb_I^+\to \Hilb_I$ is an open
    immersion at
    $(0, [I])$, hence we conclude that the image of the natural morphism
    \[
        \mathbb{A}^2 \times \Hilb_{I\subseteq I^{\sat}}^+ \to \Hilb_I
    \]
    contains an open neighbourhood of $[I]$. Then the same is true for the
    morphism
    \[
        \mathbb{A}^2 \times \Hilb_{I\subseteq I^{\sat}}\to \Hilb_I,
    \]
    which proves that $[I]$ is entirely nonsaturable.
    This proof is interesting also because while the $\TgFlag\to \Hom_{\Sk}(I,
    \Sk/I)_{0}$ is surjective, the map $\psi_{0}$ is not; only the
    map $\psi_{0, \geq 0}$ is surjective.

    We provide the \emph{Macaulay2} code below.

{\small\begin{verbatim}
S = QQ[a_0 .. a_2, Degrees=>{{1, 0}, {1, 1}, {1, 0}}];
I = ideal mingens ideal(a_0^3, a_0*a_1^2, a_0^2*a_2, a_0*a_1*a_2, a_0*a_2^4, a_2^6);
Isat = ideal mingens saturate I;
obfib = Ext^1(Isat/I, S^1/Isat);
hs = hilbertSeries(obfib, Order=>1);
T1 := (gens ring hs)_1;
assert(select(terms hs, el -> first degree el == 0) == {T1^(-1)});
assert(sum select(terms hilbertSeries(Hom(I, S^1/I), Order=>1),
    el -> first degree el == 0) == T1^6 + T1^5 + T1^4 + T1^3 + T1^2 + 4*T1 + 3 + 2*T1^(-1));
    \end{verbatim}}

Since $H_{\Sk/I^2}(9) = 17$ the fact that $I$ is not saturable could also be deduced from \cite[Theorem~1.1]{Mandziuk}.
Here, by a completely different argument, we obtain a stronger result, that $I$ is entirely nonsaturable.

\section{Wild polynomials}\label{sec:wild}
For positive integers $r$ and $n$ we denote by $\genHF{r}{n}$ the Hilbert
function of $r$ points in general position on $\mathbb{P}^n$, that is, we have
$\genHF{r}{n}(i) = \min\{r, \binom{n+i}{i}\}$ for every $i$.
By $\genMHS{r}{n}$ we denote the multigraded Hilbert scheme 
$\Hilb^{\genHF{r}{n}}$ 
and by  $\genSat{r}{n}$ we denote 
$\Satbar^{\genHF{r}{n}} \subseteq \genMHS{r}{n}$.

\begin{remark}
If $r$ and $n$ are such that $\Hilb^r(\mathbb{P}^n)$ is irreducible, then $\genSat{r}{n}$ coincides with 
the closure of the locus of saturated and radical ideals. That, is we have
$\genSat{r}{n} = \operatorname{Slip}_{r, \mathbb{P}^n}$ where $\operatorname{Slip}_{r, \mathbb{P}^n}$ is as defined in \cite{Buczyska_Buczynski__border}.
\end{remark}

In this section $\kk = \kappa$ is an algebraically closed field and $\Sk=\kappa[\alpha_0,\alpha_1,\alpha_2]$ 
is a polynomial ring with graded dual ring $\Sk^* = \kadp[x_0,x_1,x_2]$.
Given a positive integer $d$ by $\nu_d\colon \mathbb{P}((\Sk^*)_1) \to \mathbb{P}((\Sk^*)_d)$ we denote the $d$-uple embedding $\ell \mapsto \ell^{[d]}$.
We now define the relevant ranks.

\begin{definition}
Let $F \in (\Sk^*)_d$ be a nonzero form of positive degree. Given a subscheme $\Gamma$ of $\mathbb{P}(\Sk^*)_d$, 
by $\langle \Gamma \rangle$ we denote the projective linear span of $\Gamma$.
If $[F]\in \langle \Gamma \rangle$, we say that $\Gamma$ is \emph{apolar to
$F$}.
The \emph{cactus rank} $\crr(F)$ of $F$ is the smallest $r$ such that there exists
    $\Gamma\subseteq \mathbb{P}((\Sk^*)_1) = \Proj(\Sk) \simeq
    \mathbb{P}^{2}$ with $\nu_d(\Gamma)$ apolar to $F$ and of degree $r$.
    The \emph{smoothable rank} $\srr(F)$ of $F$ is defined in the same way, but
        only considering smoothable $\Gamma$,
        see~\S\ref{ssec:smoothnessOfSaturated}. Finally, the \emph{rank}
        $\rr(F)$ of
        $F$ is defined by only considering smooth $\Gamma$; for $\Gamma =
        \{[\ell_1], \ldots ,[\ell_r]\}$ we have $\langle \Gamma \rangle =
        \langle \ell_1^d,  \ldots ,\ell_r^d\rangle$. The \emph{border rank}
        $\brr(F)$ of
        $F$ is the smallest $r$ such that $F$ is a limit of rank $r$ forms.
        We have $\crr(F) \leq \srr(F) \leq \rr(F)$ and $\brr(F)\leq \srr(F)$.
        These inequalities can be strict~\cite{Bucz_Bucz_smoothable}.
\end{definition}

We assume that the characteristic of $\kappa$ is zero or larger than the degree of any form whose rank we want to compute.
Therefore, the reader may think of $\Sk^*$ as an ordinary polynomial ring,
    compare~\S\ref{ssec:macaulaysInverseSystems}.
The main result of this section is
Proposition~\ref{prop:no_wild_polynomials}, in which we show that if a divided
power polynomial $F\in (\Sk^*)_d$ satisfies  $\brr(F)\leq d+3$ then we have
$\srr(F) = \brr(F)$.

\subsection{Apolarity in three variables}

This subsection gathers some technical lemmas necessary for the
proof of Proposition~\ref{prop:no_wild_polynomials}.
Throughout we assume that $F\in (\Sk^*)_d$ is a nonzero degree $d$ form,
for $d > 0$.

\begin{lemma}\label{lem:apolar_to_wild_is_nonsaturated}
Assume that $\brr(F) \leq r < \crr(F)$ for
some positive integer $r$. Then there exists an ideal $I \subseteq F^\perp$ such that $[I] \in \genSat{r}{2}$ and $I^{\sat}_d \neq I_d$.
\end{lemma}
\begin{proof}
By the border apolarity lemma
\cite[Theorem~3.15]{Buczyska_Buczynski__border}, there is an ideal $[I] \in \genSat{r}{2}$ such that $I \subseteq F^\perp$. If $I^{\sat}_d = I_d$,
then $I^{\sat} \subseteq F^\perp$. Thus, $\crr(F) \leq r$ follows from the cactus apolarity lemma \cite[Theorem~4.7]{teitler14}.
\end{proof}

\begin{lemma}\label{lem:middle_generator}
Let $e=\lceil \frac{d+1}{2} \rceil$. For every linear form $\ell\in (\Sk)_1$,
there is an element of $(F^\perp)_e$ that is not divisible by $\ell$.
\end{lemma}
\begin{proof}
    Let $G = \ell\hook F$. Suppose that $(F^\perp)_e \subseteq (\ell)$, then
    also $(F^\perp)_{\leq e}
    \subseteq (\ell)$, so the multiplication by $\ell$ is a bijection from
    $(G^\perp)_a$ to $(F^\perp)_{a+1}$ for every $a\leq e-1$.
    Let $f$, $g$ be the Hilbert functions of
$\Sk/F^\perp$, $\Sk/G^\perp$, respectively. For $a\leq e-1$ we have
\begin{equation}\label{eq:relation_between_annihilators}
f(a+1) = \dim_{\kappa} (\Sk)_{a+1} - \dim_\kappa (F^\perp)_{a+1} = a+2 + \dim_\kappa (\Sk)_{a} -
\dim_\kappa (G^\perp)_a = a+2 + g(a). 
\end{equation}

Assume that $d$ is even. Then $d=2e-2$ and we have
\begin{align*}
f(e-2) &= f((2e-2)-(e-2)) = f(e) \stackrel{\eqref{eq:relation_between_annihilators}}{=}  g(e-1)+e+1 \\
&= g((2e-3)-(e-1))+e+1 = g(e-2) + e+1 \stackrel{\eqref{eq:relation_between_annihilators}}{=} f(e-1)+1.
\end{align*}
This contradicts the unimodality of the Hilbert function of $\Sk/F^\perp$,
see~\cite[Theorem~4.2]{stanley_hilbert_functions_of_graded_algebras}, \cite[Theorem~5.25]{iakanev}.

Assume that $d$ is odd. Then $d=2e-1$ and we have
\begin{align*}
g(e) &= g((2e-2)-e) = g(e-2) \stackrel{\eqref{eq:relation_between_annihilators}}{=} f(e-1) - e \\
&= f((2e-1)-(e-1)) -e = f(e) - e  \stackrel{\eqref{eq:relation_between_annihilators}}{=} g(e-1)+1.
\end{align*}
This contradicts the unimodality of the Hilbert function of $\Sk/G^\perp$.
\end{proof}

\begin{lemma}\label{lem:square_annihilates}
If there is a linear form $\ell$ in $(\Sk)_1$ such that $(\ell^2)\subseteq F^\perp$, then $\crr(F) \leq 2\cdot \lceil \frac{d+1}{2} \rceil$.
\end{lemma}
\begin{proof}
    Let $e=\lceil\frac{d+1}{2}\rceil$. By Lemma~\ref{lem:middle_generator}
    there exists a $\sigma\in (F^\perp)_{e}$
    not divisible by $\ell$. The ideal $(\ell^2, \sigma)$ is a complete
    intersection, hence is saturated and proves that $\crr(F)\leq 2e$.
\end{proof}

\begin{lemma}\label{lem:apolar_contained_in_line_has_special_form}
Let $k\geq 2$ be an integer. Assume that there exists $[I]\in \genSat{d+k}{2}$ such that
$I^{\sat}_1\neq 0$ and $I\subseteq F^\perp$. Then
\begin{enumerate}[label = (\arabic*)]
    \item\label{it:special_form_part_one} Up to a linear change of variables, we have
\[
I^\perp_{d+k-2} = \langle x_0G, x_1^{d+k-2}, x_1^{d+k-3}x_2, \ldots, x_2^{d+k-2}\rangle 
\]
for some $G\in \kadp[x_1,x_2]_{d+k-3}$.
\item\label{it:special_form_part_two} If $H_{\kappa[\alpha_1, \alpha_2]/G^\perp} (k-2) = k-1$, then $\alpha_0^2 \in F^\perp$.
\end{enumerate}
\end{lemma}
\begin{proof}
    The ideal $I^{\sat}$ defines a degree $d+k$ subscheme on a line, hence
    $(I^{\sat})_{\leq d+k-2} \subseteq (\ell)$ where $\ell\in (\Sk)_1$ defines the
    line. We change coordinates so that $\ell = \alpha_0$. Comparing Hilbert functions, we obtain that $\dim_{\kappa} (I^{\sat}/I)_{d+k-2}
    = 1$, so $I^\perp_{d+k-2} = \langle x_0G, x_1^{d+k-2}, x_1^{d+k-3}x_2, \ldots,
    x_2^{d+k-2}\rangle$ for some $G\in \kadp[x_0, x_1,
    x_2]_{d+k-3}$.
Since  $[I]\in \genSat{d+k}{2}$ we get from Theorem~\ref{ref:saturationDeformsWhenObFibVanishes:thm} that 
$\ObFib  = \Ext^1_{\Sk}\left(\frac{I^{\sat}}{I}, \frac{\Sk}{I^{\sat}}\right)_0\neq 0$.
It follows from Example~\ref{ref:necessaryConditionOnLine:exa} that $(\alpha_0^2)_{d+k-2} \subseteq I_{d+k-2}$. 
As a result, $I_{d+k-2}$ is as claimed in
part~\ref{it:special_form_part_one}.
If $H_{\kappa[\alpha_1, \alpha_2]/G^\perp} (k-2) = k-1$, then
the space $\mathcal{L} = x_0\cdot (\kappa[\alpha_1, \alpha_2]_{k-2}\hook G)
+ \langle x_1^d, x_1^{d-1}x_2, \ldots, x_2^d \rangle$ is $(d+k)$-dimensional.
It is contained in $I^\perp_d$, so comparing dimensions we get
$I^\perp_d = \mathcal{L}$, so $(\alpha_0^2)_d \subseteq F^{\perp}$
and part~\ref{it:special_form_part_two} follows.
\end{proof}

For $r\in \mathbb{Z}_{>0}$ let $f_r$ be the Hilbert function $(1,3,4,5, \ldots , r-1, r, r, r,
 \ldots )$.
 
\begin{lemma}\label{lem:apolar_containded_in_f}
    Assume that $d \geq 3 $
and that there exists $[I]\in \genSat{d+3}{2}$ such that $\Sk/I^{\sat}$ has Hilbert function $f_{d+3}$ and $I\subseteq F^\perp$. Up to a linear change of variables we have
$(\alpha_0^2\alpha_1, \alpha_0^2\alpha_2)\subseteq F^\perp$  or $(\alpha_0^3, \alpha_0^2\alpha_1)\subseteq F^\perp$.
\end{lemma}
\begin{proof}
We may and do assume that $I^{\sat}_2 = (\alpha_0\alpha_1, \alpha_0\alpha_2)_2$ or $I^{\sat}_2 = (\alpha_0^2,\alpha_0\alpha_1)_2$. Let $J=(\alpha_0, \theta_{d+2})$, 
where $\theta_{d+2}\in S_{d+2}$ is a minimal generator of $I^{\sat}$. It
follows from Theorem~\ref{ref:saturationDeformsWhenObFibVanishes:thm} and
Example~\ref{ex:lineAndPoints} that we have $(\alpha_0\cdot I^{\sat})_d \subseteq
I_d\subseteq (F^\perp)_d$.
From this we deduce that $\alpha_0\cdot I^{\sat}_2 \subseteq F^\perp$.
\end{proof}

\begin{lemma}\label{lem:generators_in_degree_3}
Assume that $d\geq 3$ and that there exists $[I]\in \genSat{d+3}{2}$ such that $\Sk/I^{\sat}$ has Hilbert function $\genHF{d+3}{1}$ and $I\subseteq F^\perp$. Up to a linear change of variables we have $(\alpha_0^2\ell, \alpha_0\alpha_1\ell) \subseteq F^\perp$ for a linear form  $\ell\in (\Sk)_1$. 
\end{lemma}
\begin{proof}
By
Lemma~\ref{lem:apolar_contained_in_line_has_special_form}\ref{it:special_form_part_one} applied with
$k=3$, we obtain
$I^\perp_{d+1} = \langle x_0G_d, x_1^{d+1}, x_1^dx_2, \ldots,
x_2^{d+1}\rangle$ for some nonzero $G_d\in \kadp[x_1,x_2]_d$.
Consider $\langle \alpha_1 \hook G_d, \alpha_2 \hook G_d\rangle $. If
this space is two-dimensional, then applying
Lemma~\ref{lem:apolar_contained_in_line_has_special_form}\ref{it:special_form_part_two}, we get
$\alpha_0^2\in F^\perp$. We thus assume that
$\dim_\kappa \langle \alpha_1 \hook G_d, \alpha_2 \hook G_d \rangle =
1$. We may and do assume that $G_d = x_2^d$. As a result, we have
\begin{equation}\label{eqn:no_terms_with_alpha0alpha2n-2}
I^\perp_{d-1} \supseteq  \langle x_0x_2^{d-2} \rangle.
\end{equation}

If $H_{\Sk/(\alpha_0\hook F)^\perp}(1) \leq 2$, then there is a linear form $\ell
\in (\Sk)_1$ such that $\alpha_0\ell \in F^\perp$ and the proof is complete. Thus, we assume that  $H_{\Sk/(\alpha_0\hook F)^\perp}(1) = 3$. We have
\begin{align*}
H_{\Sk/\left(F^\perp\cap (\alpha_0)\right)}(d-1) &= \dim_\kappa (\Sk)_{d-1} - \big(\dim_\kappa (\Sk)_{d-2} - H_{(\Sk)/(\alpha_0\hook F)^\perp}(d-2)\big) \\
&= d + H_{\Sk/(\alpha_0\hook F)^\perp}\big((d-1)-(d-2)\big) = d+3.
\end{align*}
It follows that $\big(F^\perp\cap(\alpha_0)\big)_{d-1} = I_{d-1}$.
Thus, from Equation \eqref{eqn:no_terms_with_alpha0alpha2n-2} we get
\[
\big(\kappa\cdot (\alpha_0\alpha_2^{d-2}\hook F)\big) \cap \big((\alpha_0^2, \alpha_0\alpha_1)_{d-1}\hook F\big) = 0.
\]
In particular, $\dim_\kappa \big((\alpha_0^2, \alpha_0\alpha_1)_{d-1}\hook F\big) \leq 2$, so there exists a linear form $\ell\in (\Sk)_1$ such that 
\[
\big((\alpha_0^2, \alpha_0\alpha_1)_{d-1}\hook F\big) \subseteq (\ell)^\perp.
\]
We get $(\alpha_0^2\ell, \alpha_0\alpha_1\ell)_d \subseteq F^\perp$ and we conclude that $\alpha_0^2\ell, \alpha_0\alpha_1\ell$ are in $F^\perp$.
\end{proof}

\subsection{Nonexistence of wild polynomials of small border rank}

A form $F\in (\Sk^*)_d$ is called \emph{wild} if its border rank is smaller than its smoothable rank.

\begin{lemma}[{\cite[Proposition~3.4]{Mandziuk}}]\label{lem:d+2_not_wild}
If $\brr(F) \leq d+2$ for some $F \in (\Sk^*)_d$, then $\crr(F) = \srr(F) = \brr(F)$. In particular, $F$ is not wild.
\end{lemma}
\begin{proof}
By \cite[Proposition~2.5]{bubu2010} we may and do assume that $\brr(F) = d+2$. Suppose that $\crr(F) > d+2$. By Lemma~\ref{lem:apolar_to_wild_is_nonsaturated} there exists $[I]\in \genSat{d+2}{2}$ such that $I\subseteq F^\perp$ and $I^{\sat}_d\neq I_{d}$. Then necessarily, $\Sk/I^{\sat}$ has Hilbert function $\genHF{d+2}{1}$ and it follows from Lemma~\ref{lem:apolar_contained_in_line_has_special_form}(2) that the square of a linear form annihilates $F$. From Lemma~\ref{lem:square_annihilates} we get $\crr(F) \leq 2\lceil\frac{d+1}{2}\rceil \leq d+2$.
\end{proof}

\begin{lemma}\label{lem:special_case_3}
Let $d\geq 4$ and $F_d\in (\Sk^*)_d$ be of the form $F_d = ax_0x_1x_2^{d-2} +
G_d(x_0,x_2) + H_d(x_1,x_2)$ for some $a\in \kappa$ and homogeneous, degree $d$
divided power polynomials $G_d$ and $H_d$.
The cactus rank of $F_d$ is at most $d+3$.
\end{lemma}
\begin{proof}
We have $\crr(G_d) \leq \lceil \frac{d+1}{2} \rceil$ so we may and do assume that $a=1$.
Assume first that $d$ is even and let $d=2k$. We consider two cases:
\begin{enumerate}[label=(\arabic*)]
\item\label{it:special_case_3_part_one} $((\alpha_0^2\hook G_d)^\perp)_{k-1}
    \neq 0$ and $((\alpha_1^2\hook H_d)^\perp)_{k-1} \neq 0$;
\item\label{it:special_case_3_part_two} $((\alpha_0^2\hook G_d)^\perp)_{k-1} =
    0$ or $((\alpha_1^2\hook H_d)^\perp)_{k-1} = 0$.
\end{enumerate}
In the first case, there are $\theta_{k-1}(\alpha_0, \alpha_2)$ and $\eta_{k-1}(\alpha_1, \alpha_2)$ such that $I=(\alpha_0^2\alpha_1, \alpha_0\alpha_1^2, \alpha_0^2\theta_{k-1}, \alpha_1^2\eta_{k-1}) \subseteq F^\perp$. Its initial ideal with respect to any term order is of the form $J=(\alpha_0^2\alpha_1, \alpha_0\alpha_1^2, \alpha_0^2\theta'_{k-1}, \alpha_1^2\eta'_{k-1})$ for some monomials $\theta'\in \kappa[\alpha_0,\alpha_2]_{k-1}$ and $\eta' \in \kappa[\alpha_1, \alpha_2]_{k-1}$. The ideal $J$ is saturated and the Hilbert polynomial of $\Sk/J$ is $2k+2 = d+2$. Therefore, $\crr(F) \leq d+2$ follows.

If \ref{it:special_case_3_part_two} holds, we may and do assume that
$((\alpha_0^2\hook G_d)^\perp)_{k-1} = 0$. The map 
\[
\kappa[\alpha_0,\alpha_2]_{k-1} \to \kadp[x_0,x_2]_{d-k-1} = \kadp[x_0,x_2]_{k-1} \text{ given by }\theta \mapsto (\alpha_0^2\theta) \hook G_d
\]
is injective and thus bijective.
It follows that for some $\theta_{k-1}\in \kappa[\alpha_0,\alpha_2]_{k-1}$, $\eta_{k-1}\in \kappa[\alpha_1, \alpha_2]_{k-1}$ and $c\in \{0,1\}$, the ideal
\[
I = (\alpha_0^2\alpha_1, \alpha_0\alpha_1^2, \alpha_0^2\theta_{k-1} - \alpha_0\alpha_1\alpha_2^{k-1}, \alpha_1^2\eta_{k-1} - c\alpha_0\alpha_1\alpha_2^{k-1})
\]
is contained in $F^\perp$. If $c=0$ or $\theta_{k-1}\notin \kappa\cdot\alpha_2^{k-1}$ or $\eta_{k-1}\notin \kappa\cdot \alpha_2^{k-1}$, then these generators form a Gr\"obner basis with respect to the grevlex order with $\alpha_0>\alpha_1>\alpha_2$ or $\alpha_1 > \alpha_0 > \alpha_2$ and the proof is as in case~\ref{it:special_case_3_part_one}. Therefore, we are left with the case that 
$ 
I = (\alpha_0^2\alpha_1, \alpha_0\alpha_1^2, \alpha_0^2\alpha_2^{k-1} -\alpha_0\alpha_1\alpha_2^{k-1}, c'\alpha_1^2\alpha_2^{k-1}-\alpha_0\alpha_1\alpha_2^{k-1}) \subseteq F^\perp
$ for some nonzero $c'$. This ideal contains $I'=(\alpha_0^2\alpha_1, \alpha_0\alpha_1^2, \alpha_0^2\alpha_2^{k-1} -c'\alpha_1^2\alpha_2^{k-1}) = (\alpha_0^2\alpha_1, \alpha_0\alpha_1^2, \alpha_0^2\alpha_2^{k-1} -c'\alpha_1^2\alpha_2^{k-1}, \alpha_1^3\alpha_2^{k-1})$. By taking an initial ideal, one may verify that $I'$ is saturated and $\Sk/I'$ has Hilbert polynomial $2k+3 = d+3$.

Now assume that $d$ is odd and let $d=2k+1$. Then for dimension reasons, there are $\theta_{k}\in \kappa[\alpha_0,\alpha_2]_k$ and $\eta_{k}\in \kappa[\alpha_1, \alpha_2]_k$ such that $I=(\alpha_0^2\alpha_1, \alpha_0\alpha_1^2, \alpha_0^2\theta_{k}, \alpha_1^2\eta_k)\subseteq F^\perp$. By taking an initial ideal, one may verify that $I$ is saturated and $\Sk/I$ has Hilbert polynomial $2k+4=d+3$.
\end{proof}

\begin{lemma}\label{lem:special_case_4}
Let $d$ be a positive integer and $F_d\in (\Sk^*)_d$ be of the form $F_d = ax_0x_2^{d-1} + G_d(x_0,x_1) + H_d(x_1,x_2)$
for some $a\in \kappa$ and homogeneous, degree $d$ divided power polynomials $G_d$ and $H_d$.
The cactus rank of $F_d$ is at most $d+3$.
\end{lemma}
\begin{proof}
    Let $e = \lceil \frac{d+1}{2} \rceil$.
The cactus rank of $G_d$ is at most $e$, so it suffices to show that the
cactus rank of $x_0x_2^{d-1}+H_d(x_1,x_2)$ is at most $e+1$.
For some $\theta \in \kappa[\alpha_1,\alpha_2]_{e-1}$ and $a\in \{0,1\}$, we
have $I=(\alpha_0^2, \alpha_0\alpha_1, \alpha_1\theta-a\alpha_0\alpha_2^{e-1})
\subseteq (x_0x_2^{d-1}+H_d(x_1,x_2))^\perp$. The ideal $I$ is saturated and
$\Sk/I$ has Hilbert polynomial $e+1$.
\end{proof}

\begin{proposition}\label{prop:no_wild_polynomials}
Assume that $\brr(F) \leq d+3$ for some $F\in (\Sk^*)_d$. If the characteristic of $\kappa$ is zero or larger than $d$, then $\crr(F) = \srr(F) = \brr(F)$. In particular, $F$ is not wild.
\end{proposition}
\begin{proof}
By Lemma~\ref{lem:d+2_not_wild} we may and do assume that $\brr(F) = d+3$.
By the Alexander-Hirschowitz theorem, see \cite{AlexHirsch} or
\cite[Theorem~1.61]{iakanev}, for $d\leq 5$ such a high border rank does not
occur,
hence we obtain $d\geq 6$.
Suppose that $\crr(F) > d+3$. By Lemma~\ref{lem:apolar_to_wild_is_nonsaturated} there  is an ideal $I\subseteq F^\perp$ such that $I^{\sat}_d\neq I_d$ and $[I]\in \genSat{d+3}{2}$. There are two possibilities for the Hilbert function of $\Sk/I^{\sat}$. It is either $\genHF{d+3}{1}$ or $f_{d+3}$. It follows from Lemmas~\ref{lem:apolar_containded_in_f} and \ref{lem:generators_in_degree_3} that we may assume that one of the following holds
\begin{enumerate}[label=(\arabic*)]
\item\label{it:wild_polynomials_part_one} $(\alpha_0^2\alpha_1, \alpha_0^2\alpha_2)\subseteq F^\perp$;
\item\label{it:wild_polynomials_part_two} $(\alpha_0^3, \alpha_0^2\alpha_1)\subseteq F^\perp$;
\item\label{it:wild_polynomials_part_three} $(\alpha_0^2\alpha_1, \alpha_0\alpha_1^2) \subseteq F^\perp$;
\item\label{it:wild_polynomials_part_four} $(\alpha_0^2\alpha_2, \alpha_0\alpha_1\alpha_2)\subseteq F^\perp$.
\end{enumerate}
Let $e = \lceil\frac{d+1}{2}\rceil$.  Let $<$ be the grevlex monomial order with $\alpha_0<\alpha_1<\alpha_2$.

Assume that \ref{it:wild_polynomials_part_one} holds. Let $\eta_e\in (F^\perp)_e$ be not divisible by $\alpha_0$ (see Lemma~\ref{lem:middle_generator}). We choose one of the form  $\eta_e = \zeta_e + a\alpha_0^e + \alpha_0\alpha_1\zeta_{e-2} + \alpha_0\alpha_2\zeta'_{e-2}$, where $a\in \kappa$, $\zeta_e \in \kappa[\alpha_1,\alpha_2]_{e}\setminus \{0\}$ and $\zeta_{e-2}, \zeta'_{e-2}\in \kappa[\alpha_1,\alpha_2]_{e-2}$. Furthermore, if $a\neq 0$, then $\alpha_0^{e+2} \in F^\perp$. It follows that $\alpha_0^2\in F^\perp$, so by Lemma~\ref{lem:square_annihilates} the cactus rank of $F$ is at most $d+2$.
Let $J = \operatorname{in}_<(\alpha_0^2\alpha_1, \alpha_0^2\alpha_2, \eta_e)$. Then $J= (\alpha_0^2\alpha_1, \alpha_0^2\alpha_2, M)$ where $M$ is a monomial in $\kappa[\alpha_1,\alpha_2]_e$. It follows that $\Sk/J$ has Hilbert polynomial $2e+1\leq d+3$ and $J$ is saturated.

Assume that \ref{it:wild_polynomials_part_two} holds. Let $\eta_e\in (F^\perp)_e$ be not divisible by $\alpha_0$ (see Lemma~\ref{lem:middle_generator}). We choose one of the form $\eta_e = a\alpha_2^e + \alpha_1\zeta_{e-1} + b\alpha_0^2\alpha_2^{e-2} + \alpha_0\zeta'_{e-1}$, where $a,b\in \kappa$, $\zeta_{e-1}, \zeta_{e-1}' \in \kappa[\alpha_1,\alpha_2]_{e-1}$ and $a\alpha_2^2+\alpha_1\zeta_{e-1} \neq 0$. If $a\neq 0$, then $\alpha_0^2\alpha_2^e \in F^\perp$ and it follows that $\alpha_0^2\in F^\perp$. Then by Lemma~\ref{lem:square_annihilates} we get $\crr(F) \leq d+2$. Therefore, $J=\operatorname{in}_<(\alpha_0^3, \alpha_0^2\alpha_1, \eta_e) = (\alpha_0^3, \alpha_0^2\alpha_1, \alpha_1M)$ for a monomial $M\in \kappa[\alpha_1,\alpha_2]_{e-1}$. It follows that $J$ is saturated and $\Sk/J$ has Hilbert polynomial $2e+1 \leq d+3$.

If \ref{it:wild_polynomials_part_three} or \ref{it:wild_polynomials_part_four} hold, then $\crr(F) \leq d+3$ by Lemma~\ref{lem:special_case_3} or \ref{lem:special_case_4}, respectively.
\end{proof}

\appendix

\section{Obstruction theory of flag Hilbert schemes}
Throughout this appendix, we keep the assumptions from~\S\ref{ssec:setup}, in
particular we fix
$\kk \to \kappa$, where $\kappa$ is a field, and
take homogeneous ideals $K \subseteq J$ in the
$\AA$-graded algebra $\Sk$. Arguing as in for example~\cite[Theorem
6.4.5]{fantechi_et_al_fundamental_ag}
or~\cite[Proposition~4.1]{Jelisiejew__Elementary}, we get that the scheme
$\Hilb_K$ has obstruction theory with obstruction group $\Ext^1_{\Sk}(K,
\frac{\Sk}{K})_{\zerodeg}$.

In the case $\kk= \kappa$, the basic results of deformation theory are phrased
in our language in Chapter~5 of~\cite{fantechi_et_al_fundamental_ag}
or~\cite{Fantechi_Manetti}. In the general case, it is harder to find the
reference for the following well-known result, so we sketch a proof.

\begin{lemma}[Main Theorem of Obstruction
    Calculus]\label{ref:obstructionCalculus:lem}
    Let $Y$ be a Noetherian scheme.
    Suppose that $f\colon (X, x)\to (Y, y)$ is a finitely-presented morphism
    of pointed schemes that yields an isomorphism on residue fields $\kappa(y)
    \simeq \kappa(x)$.
    Assume that $(X, x)$ and $(Y,y)$ admit obstruction
    theories, that $f$ is a map of obstruction theories and that
    \begin{enumerate}
        \item $df|_x\colon T_x X \to T_y Y$ is surjective,
        \item the map of obstruction groups induced by $f$ is injective.
    \end{enumerate}
    Then $f$ is smooth at $x$.
\end{lemma}
\begin{proof}
    The surjectivity of $df|_x$ and injectivity of map on obstruction groups
    imply that for any small extension $B\to A\to 0$ of Artinian local rings
    and any diagram
    \[
        \begin{tikzcd}
            (\Spec(A), *)\ar[d, hook, "cl"] \ar[r] & (X,x)\ar[d, "f"]\\
            (\Spec(B), *) \ar[r]  &(Y, y)
        \end{tikzcd}
    \]
    of pointed schemes, there is a lift $(\Spec(B), *)\to (X, x)$. Having
    this, the
    claim follows
    from~\cite[Proposition~1.1, Remark 1.3(3)]{Artin_theorems_of_representability}.
    For an alternative, one can argue using~\cite[Proposition~(17.5.3)]{ega4-4}.
\end{proof}

We pass to constructing the obstruction space for the flag Hilbert scheme.
As in~\S\ref{ssec:maindiagram}, we
have natural maps $\Hom_{\Sk}(K, \Sk/K), \Hom_{\Sk}(J, \Sk/J)\to \Hom_{\Sk}(K,
\Sk/J)$. Denote
their sum as
\[
    \psi\colon \Hom_{\Sk}(K, \Sk/K) \oplus \Hom_{\Sk}(J, \Sk/J)\to
    \Hom_{\Sk}(K, \Sk/J).
\]
Recall that $\Hilb_{K\subseteq J}\to \Spec(\kk)$ is the flag Hilbert scheme
and that $K\subseteq J\subseteq \Sk$ yields its $\Spec(\kappa)$-point. Below
we restrict to $\kk$-schemes. Therefore, $\kk$ becomes a base scheme and
does not appear in the argument, in particular it is not important that it may
not be a field.
\begin{theorem}\label{ref:obstructionAtFlag:thm}
    Suppose that $\psi_{\zerodeg}$ is surjective. Then the pointed $\kk$-scheme
    $\Hilb_{K\subseteq J}$ has an obstruction theory with
    obstruction group $\ObFlag$ and the natural projections
    $\ObFlag\to \Ext^1_{\Sk}(K, \frac{\Sk}{K})_{\zerodeg}$, $\ObFlag\to
    \Ext^1_{\Sk}(J,
    \frac{\Sk}{J})_{\zerodeg}$ are
    maps of obstruction theories for the maps $\pr_K$ and $\pr_J$, respectively.
\end{theorem}
This theorem partially generalizes~\cite[Theorem~4.10]{Jelisiejew__Elementary}, which
has stronger assumptions and is formulated only in the case $\kk = \kappa$,
however considers the case ``$\geq \zerodeg$'' rather that degree zero.
\begin{proof}
    \def\II{\mathcal{I}}%
    \def\IIbar{\overline{\mathcal{I}}}%
    \def\QQ{\mathcal{Q}}%
    \def\QQbar{\overline{\mathcal{Q}}}%

    A large part of the proof follows as in
    Theorem~{\cite[Theorem~4.10]{Jelisiejew__Elementary}}. We only sketch this
    part, referring the reader to the above mentioned proof for details.
    The obstruction theories of $\Hilb_K$ and $\Hilb_J$ yield an obstruction theory for
    $\Hilb_K\times \Hilb_J$ and by direct check, the obstruction at a point of
    $\Hilb_{K\subseteq J}$ lies in the subgroup $\ObFlag$, compare~\cite[Theorem~4.10]{Jelisiejew__Elementary}. None of this uses
    the assumption that $\psi_{\zerodeg}$ is surjective.

    What is left to prove is that if the obstruction vanishes, an extension
    exists. Let us explain this more precisely: we consider a small extension $\Spec(A)\into
    \Spec(B)$ of local Artinian $\kk$-algebras with residue field $\kappa$ and
    a map $\Spec(A)\to \Hilb_{K\subseteq J}$.
    If the obstruction
    vanishes, then we want to prove that there exists a map from $\Spec(B)$ to $\Hilb_{K\subseteq
    J}$ extending a given one from $\Spec(A)$.
    A small extension can be composed of the ones when
    $\ker(B\to A)$ has length one (so it is annihilated by the maximal ideal
    of $B$), so we assume that our extension is
    \[
        0\to \kappa \to B\to A\to 0.
    \]
    We denote by $\varepsilon\in B$ the image of $1$ in $B$.
    By vanishing of the
    obstruction for $\Hilb_K\times \Hilb_J$, there exists a commutative
    diagram
    \[
        \begin{tikzcd}
            \Spec(A) \ar[r]\ar[d, "cl"] & \Hilb_{K\subseteq J}\ar[d, "cl"]\\
            \Spec(B) \ar[r] & \Hilb_K \times \Hilb_J.
        \end{tikzcd}
    \]
    We are to prove that there \emph{exists} such a diagram, with $\Spec(B)\to
    \Hilb_K\times \Hilb_J$ actually factoring through $\Hilb_{K\subseteq J}$.
    We stress the word \emph{exists}: in the course of the proof we will
    modify the map $\Spec(B)\to \Hilb_{K} \times \Hilb_J$ in the diagram
    above.

    \begin{figure}[h]
        \centering
        \begin{tikzcd}
            & K \ar[rr, hook]\ar[dd, hook] && \Sk\ar[rr, two heads]\ar[dd, hook,
            "\cdot\varepsilon" near start] && \Sk/K\ar[dd, hook]\\
            J \ar[rr, crossing over, hook]\ar[from=ru, hook] && \Sk  \ar[from=ru, equal]\ar[rr, crossing over, two heads] && \Sk/J
            \ar[from=ru, two heads] & \\
            & \II_K \ar[rr, hook]\ar[dd, two heads] && S_B\ar[rr, two heads]\ar[dd, two heads] &&
            \QQ_K\ar[dd, two heads]\\
            \II_J \ar[rr, crossing over, hook]\ar[from=uu, crossing over,
            hook]
            && S_B  \ar[ru, equal]\ar[rr, crossing over, two
            heads]\ar[from=uu, crossing over, hook, "\cdot\varepsilon" near start] &&
            \QQ_J\ar[from=uu, crossing over, hook]
            & \\
            & \IIbar_K \ar[rr, hook] && S_A\ar[rr, two heads] && \QQbar_K\\
            \IIbar_J\ar[from=uu, crossing over, two heads] \ar[rr,hook]\ar[from=ru, hook]
            && S_A \ar[ru, equal]\ar[rr, two heads]\ar[from=uu, crossing over, two heads]
            && \QQbar_J
            \ar[from=ru, two heads]\ar[from=uu, crossing over, two heads] & 
        \end{tikzcd}
        \caption{Deformations of $K$ and $J$, but not yet a deformation of
        a pair $K\subseteq J$.}
        \label{eq:diagDef}
    \end{figure}

    Let $\II_K, \II_J\subseteq S_B$ be the ideals of the deformations
    corresponding to $\Spec(B)\to \Hilb_K\times \Hilb_J$. Let $\star\in\{K,
    J\}$ and let
    $\IIbar_{\star} := \II_{\star}\tensor_B A$ be the restrictions to
    $A$. By assumption we have $\IIbar_K\subseteq \IIbar_J$.
    Let $\QQ_{\star} := S_B/\II_{\star}$ be the quotients and
    $\QQbar_{\star} = S_A/\IIbar_{\star}  \simeq \QQ_{\star}\tensor_B
    A$. We have the commutative diagram as in Figure~\ref{eq:diagDef}.

    From its middle horizontal plane we obtain a homomorphism
    $f\colon \II_K\to S_B\to \QQ_J$. The composition of $f$ with $\QQ_J\onto
    \QQbar_J$ is zero by a diagram chase, since $\IIbar_K\subseteq \IIbar_J$.
    Thus $f$ lifts to $\II_K\to \varepsilon(\Sk/J)$ and hence gives a
    homomorphism
    $\bar{f}\in \Hom_{\Sk}(K, \Sk/J)_{\zerodeg}$. By assumption on
    $\psi_{\zerodeg}$, we can write $\bar{f} = -\overline{\varphi_K} +
    \overline{\varphi_J}$, where $\varphi_K\in \Hom_{\Sk}(K, \Sk/K)_{\zerodeg}$ and
    $\varphi_J\in \Hom_{\Sk}(J, \Sk/J)_{\zerodeg}$.

    Now, recall that $\Spec(B)\to \Hilb_K\times \Hilb_J$ was
    only \emph{an} extension of
    $\Spec(A)$. The other possible extensions admit an action of the vector space $\Hom_{\Sk}(K,
    \Sk/K)_{\zerodeg}\oplus\Hom_{\Sk}(J, \Sk/J)_{\zerodeg}$
    by the argument of~\cite[Theorem~6.4.5]{fantechi_et_al_fundamental_ag}.
    We claim that
    by acting on the above extension with $(\varphi_K,\varphi_J)$ we obtain an
    extension that satisfies $\bar{f}=0$.
    To prove this, we need to unpack the abstract statements in the previous
    paragraph. For $s\in K$ let $s_A\in \IIbar_K\subseteq \IIbar_J$ be its lift and $s_K\in
    \II_K$ and $s_J\in \II_J$ be lifts of $s_A$. Let $\delta_s\in \Sk$ be such
    that $s_K - s_J = \delta_s \varepsilon$.
    Then the map $\bar{f}$ lifts an element $s$ to $s_K$, maps it to
    $S_B/\II_J$ and then lifts the resulting class to $\Sk/J$. One such lift is the
    class of $\delta_s$ in $\Sk/J$. So $\bar{f}(s) = \delta_s\mod J$ and we
    obtain that $\delta_s \mod J = -\varphi_K(s) + \varphi_J(s)$.

    Let $\varphi_K'$ be the composition
    \[
        \II_K\onto \II_K\tensor_B \kappa \simeq K\to \Sk/K  \simeq
        \varepsilon\cdot \Sk/K
        \into S_B/\varepsilon K,
    \]
    where the middle map is $\varphi_K$. Let $\II_K'$ be the preimage in $S_B$
    of
    \[
        \im(\id_{\II_K} + \varphi_K') = (i +
        \varphi_K'(i)\ |\ i\in \II_K) \subseteq S_B/\varepsilon K.
    \]
    Unravelling the proof
    of~\cite[Theorem~6.4.5(b)]{fantechi_et_al_fundamental_ag} or by a direct
    check with~\cite[Lemma~6.4.3]{fantechi_et_al_fundamental_ag} we see that
    $S_B/\II_K'$ is flat over $B$.
    Similarly, from $\varphi_J$ we obtain $\varphi_J'$, $\II_J'$ and $\QQ_J' := S_B/\II_J'$.
    From $(\II_K', \II_J')$ we obtain a map $\bar{f}'\colon K\to \Sk/J$ by the
    same procedure as for $\bar{f}$ above. Let us analyse it.
    Let us keep the notation $s$, $s_A$, $s_K$, $s_J$ from above. For an
    element of $\Sk/K$ by
    $\widehat{\cdot}$ we denote any its lift to $\Sk$.
    The elements $s_K' = s_K +
    \varepsilon \widehat{\varphi_K(s)}\in \II_K'$ and $s_J' = s_J + \varepsilon
    \widehat{\varphi_J(s)}\in \II_J'$ are lifts of $s_A$.
    We have
    \[
        s_K' - s_J' = s_K + \varepsilon\widehat{\varphi_K(s)} - s_J -
        \varepsilon\widehat{\varphi_J(s)} =
        \varepsilon(\delta_s + \widehat{\varphi_K(s)} -
        \widehat{\varphi_J(s)}),
    \]
    so we can take $\delta_s' = \delta_s + \widehat{\varphi_K(s)} -
        \widehat{\varphi_J(s)}$. But then $\delta_s' \equiv 0 \mod J$,
    which proves that $\bar{f}' = 0$. This implies that $\II_K'\subseteq
    \II_J'$ as claimed. The map $\Spec(B)\to \Hilb_{K\subseteq J}$
    corresponding to $(\II_K', \II_J')$ lifts the given map $\Spec(A)$. This concludes the
    proof.
\end{proof}


\begin{lemma}\label{ref:abstractThy:lem}
    Let $f\colon R_1\onto R_2$ be a surjective homomorphism of local Noetherian
    $\kk$-algebras with residue fields $\kappa$. Suppose that for every small
    extension $B\onto A$ of Artinian local $\kk$-algebras with residue field $\kappa$ and every commutative diagram of
    affine $\kk$-schemes
    \[
        \begin{tikzcd}
            &\Spec(R_2)\ar[from=d, "p"']\ar[r, "f^{\#}", hook, "cl"'] &
            \Spec(R_1)\ar[from=d]\\
            \Spec(\kappa)\ar[ru, hook]\ar[r, hook] & \Spec(A) \ar[r, hook] & \Spec(B)
        \end{tikzcd}
    \]
    there exists $p'\colon \Spec(B) \to \Spec(R_2)$ such that $p$ is the
    restriction of $p'$ (but we do not assume any compatibility with $f^{\#}$!).
    Then $\ker(f)$ is generated by
    $d$ elements whose classes are independent in $\mm_{R_1}/\mm_{R_1}^2$,
    where $d = \dim_{\kappa} \mm_{R_1}/\mm_{R_1}^2 - \dim_{\kappa}
    \mm_{R_2}/\mm_{R_2}^2$.
\end{lemma}
\begin{proof}
    \def\nn{\mathfrak{n}}%
    Let $d = \dim_{\kappa} (\mm_{R_1}^2 + \ker f)/\mm_{R_1}^2$ and let $K$ be
    generated by any $d$ elements of $\ker f$ whose classes span this space.
    Let $\overline{R}_1 := R_1/K$. It follows that the homomorphism
    $\pi\colon \overline{R}_1\to R_2$ also satisfies the assumptions of the statement
    and moreover induces an isomorphism on cotangent spaces.
    Let $\nn = \mm_{\overline{R}_1}$.
    Suppose that $\pi$ is not an isomorphism and
    let $s\geq 3$ be the minimal element such that $\ker \pi \not \subseteq
    \nn^s$.
    Let $A = \overline{R}_1/\nn^{s-1}$ and $B = \overline{R}_1/\nn^{s}$. By
    assumption, the canonical surjective homomorphism $R_2\to A$ lifts to a homomorphism $g\colon R_2\to B$. We obtain a
    diagram of homomorphisms of $\kk$-algebras
    \[
        \begin{tikzcd}
            R_2 = \overline{R}_1/ \ker \pi\ar[d, two heads, "can"]\ar[rd, "g"]\\
            \overline{R}_1 / \nn^{s-1} & \ar[l, two heads, "can"']
            \overline{R}_1 / \nn^{s}
        \end{tikzcd}
    \]
    The canonical maps induce isomorphisms of cotangent spaces, hence so does
    $g$. But then $g$ is surjective. Yet the map $g$ factors through $\overline{R}_1/(\ker \pi + \nn^s)\to
    \overline{R}_1/\nn^s$ which cannot be surjective, as
    $\dim_{\kappa}\overline{R}_1/(\ker \pi + \nn^s) < \dim_{\kappa}
    \overline{R}_1/\nn^s$. The contradiction shows that $\ker \pi = 0$ and
    the proof is completed.
\end{proof}

To apply Lemma~\ref{ref:abstractThy:lem} to the Hilbert schemes, we need to
understand the full tangent spaces, not only the tangent spaces to the fiber.
\begin{lemma}[image in the tangent to the base]\label{ref:sernesi:lem}
    Let $H\to \Spec(\kk)$ be one of the Hilbert schemes considered above, let
    $\ObSymbol$ be its obstruction group and
    let $x\colon \Spec(\kappa)\to H$ be a point. Then the obstruction map yields a
    $\kappa$-linear map $T_{\kappa} \Spec(\kk) \to \ObSymbol$ and the kernel of this
    map is the image of $T_{x}H$ in $T_{\kappa}\Spec(\kk)$. If $f\colon H\to H'$ is a
    morphims of such Hilbert schemes which induces an injection on obstruction
    groups, then the images of $T_{x} H$ and $T_{f(x)} H'$ in $T_{\kappa}
    \Spec(\kk)$ coincide.
\end{lemma}
\begin{proof}
    This follows as in~\cite[Proposition~4.4.4]{Sernesi__Deformations}. The
    key point is that a tangent vector in $T_{\kappa} \Spec(\kk)$ yields a map
    $\Spec(\kappa[\varepsilon]/\varepsilon^2)\to \Spec(\kk)$ and this map
    lifts to a map to $H$ if and only if the associated obstruction in
    $\ObSymbol$ vanishes.
\end{proof}

\begin{corollary}
    Under the assumptions of Theorem~\ref{ref:obstructionAtFlag:thm} the
    subscheme $\Hilb_{K\subseteq J} \subseteq \Hilb_K \times \Hilb_J$ is cut
    out near $[K\subseteq J]$ by $\dim_{\kappa} \Hom(K,
    \Sk/J)_{\zerodeg}$
    equations which yield linearly independent classes in the cotangent space
    at $[K\subseteq J]$.
\end{corollary}
We stress that without additional assumptions on $\Hilb_K\times \Hilb_J$ this
does not imply that $\Hilb_{K\subseteq J}$ has codimension $\dim_{\kappa} \Hom(K,
\Sk/J)_{\zerodeg}$.
\begin{proof}
    Theorem~\ref{ref:obstructionAtFlag:thm} shows that closed immersion
    induces an injective map on obstruction spaces. This by definition of
    obstruction implies that the assumptions of
    Lemma~\ref{ref:abstractThy:lem} are satisfied for
    \[
        R_1 = \OO_{\Hilb_K\times \Hilb_J, ([K], [J])} \onto R_2 =
        \OO_{\Hilb_{K\subseteq J}, [K\subseteq J]}.
    \]
    It remains to understand the difference of dimensions of tangent maps. The images of both tangent spaces in
    $T_{\kappa}\Spec(\kk)$ coincide by Lemma~\ref{ref:sernesi:lem}. So we can restrict to
    tangent spaces to the fibers.
    Finally, since $\psi_{\zerodeg}$ is surjective, we see from the main
    diagram~\S\ref{ssec:maindiagram} that
    \[
        \dim_{\kappa} T_{([K], [J])} \Hilb_K \times \Hilb_J - \dim_{\kappa}
        T_{[K\subseteq J]} \Hilb_{K\subseteq J} = \dim_{\kappa} \Hom_{\Sk}(K,
        \Sk/J)_{\zerodeg}.\qedhere
    \]
\end{proof}

\newcommand{\etalchar}[1]{$^{#1}$}

\end{document}